\documentclass[12pt,twoside]{amsart}

\usepackage[colorlinks,
linkcolor=red,
citecolor=blue
]{hyperref}

\usepackage{xcolor}

\def\cX{{\mathcal X}}
\newtheorem{theorem}{Theorem}[section]  
\newtheorem{proposition}{Proposition}[section]  
  
\newtheorem{Lemma}{Lemma}[section] 
  
\newtheorem{remark}{Remark}[section]  
\newtheorem{cor}{Corollary}[section]  
\usepackage{amsmath,amsfonts,amssymb}
\usepackage{graphicx}    

\def\dd'{\dot \Delta_{j'}}

\def\var{\varepsilon}
\def\tL{\tilde{L}}
\def\dB{\dot{B}}

\def\tilde{\widetilde}
\def\hat{\widehat}

\newcommand\R{\mathbb{R}}

\newcommand\Z{\mathbb{Z}}

\newcommand{\Tr}{\hbox{\rm{Tr}\,}}

\renewcommand{\div}{\mbox{\rm div}\;\!}

\newcommand{\curl}{\mbox{\rm curl}\,}

\def\cC{{\mathcal C}}

\def\cD{{\mathcal D}}
\def\cE{{\mathcal E}}

\def\cY{{\mathcal Y}}

\def\tilde{\widetilde}
\def\hat{\widehat}

\begin{document}

\title[The pressureless damped Euler-Riesz system]{The pressureless damped Euler-Riesz system in the critical regularity framework}

\author{Meiling Chi}
\address{School of Mathematics, Nanjing University of Aeronautics and Astronautics, Nanjing 211106, P. R. China}
\email{chiml@nuaa.edu.cn}

\author{Ling-Yun Shou}
\address{School of Mathematics, Nanjing University of Aeronautics and Astronautics, Nanjing 211106, P. R. China}
\email{shoulingyun11@gmail.com}

\author{Jiang Xu}
\address{School of Mathematics, Nanjing University of Aeronautics and Astronautics, Nanjing 211106, P. R. China}
\email{jiangxu\_79@nuaa.edu.cn}


\begin{abstract} 
	
 We are concerned with a system governing the evolution of the pressureless compressible Euler equations with Riesz interaction and damping in $\mathbb{R}^{d}$ ($d\geq1$), where the interaction force is given by $\nabla(-\Delta)^{\smash{\frac{\alpha-d}{2}}}(\rho-\bar{\rho})$ with $d-2<\alpha<d$. Referring to the standard dissipative structure of first-order hyperbolic systems, the purpose of this paper is to investigate the weaker dissipation effect arising from the interaction force and to establish the global existence and large-time behavior of solutions to the Cauchy problem in the critical $L^p$ framework. More precisely, it is observed by the spectral analysis that the density behaves like fractional heat diffusion at low frequencies. Furthermore, if the low-frequency part of the initial perturbation is bounded in some Besov space $\dot{B}^{\sigma_1}_{p,\infty}$ with $-d/p-1\leq \sigma_1<d/p-1$, it is shown that the $L^p$-norm of the $\sigma$-order derivative for the density converges to its equilibrium at the rate $(1+t)^{-\smash{\frac{\sigma-\sigma_1}{\alpha-d+2}}}$, which coincides with that of the  fractional heat kernel.

	\end{abstract}
 
 \maketitle
\section{Introduction}

The motions of nonlocal interactions appear in many applications in the field of condensed matter physics, plasma physics and collective dynamics in biology \cite{ bz,fef, hg}. In this paper, we consider the pressureless compressible Euler-Riesz system with drag forces in $\R^d$ ($d\geq1$), which takes the form
\begin{equation}\label{E-R}
	\left\{\begin{array}{l}
		\partial_t\rho+\div(\rho u)=0,\\
		\partial_t(\rho u)+\div(\rho u\otimes u)
		=-\lambda\rho u-\kappa \rho\nabla\Lambda^{\alpha-d}(\rho-\bar{\rho}),
	\end{array}\right.
\end{equation}
where $\rho=\rho(t,x)\geq0$ and $u=u(t,x)\in\mathbb{R}^{d}$ denote the density and velocity of the fluid at the time $t$ and the position $x$, respectively. In addition, $\lambda>0$ is the damping coefficient, $\kappa>0$ stands for the coefficient representing the strength of the interaction force, and $\Lambda^{\alpha-d}=(-\Delta)^{\smash{\frac{\alpha-d}{2}}}$ denotes the fractional Laplacian operator with $d-2<\alpha<d$. 
Note that the case $\alpha=d-2$ corresponds to Coulomb interaction, and we refer to the range $d-2<\alpha<d$ as the Riesz interaction.

We investigate the Cauchy problem of the system \eqref{E-R} subject to the initial data 
\begin{equation}\label{d}
(\rho,u)(0,x)=(\rho_0,u_0)(x),\quad\quad x\in\mathbb{R}^{d} 
\end{equation}
 with the far-field behavior
 \begin{equation}
 (\rho_0,u_0)(x)\rightarrow (\bar{\rho},0)\quad\text{as}\quad |x|\rightarrow\infty,\nonumber
 \end{equation}
 where $\bar{\rho}>0$ is the constant background density.

The pressureless Euler-Riesz system \eqref{E-R} arises from the intricate particle dynamics corresponding to Newton's laws, where the non-local interactions among particles are characterized by Riesz potentials.  Serfaty \cite{Serfaty} derived the pressureless Euler-Riesz system from the second-order particle system in the sense of  \emph{mean field limits}. In our setting, the frictional effect is incorporated into the model to govern and ensure the system’s stability (cf. \cite{choi and jung}). Then, as in Serfaty \cite{Serfaty}, the dynamics of $N$  particles through with the interaction force $\nabla_{x}K$ is described as follows:
\begin{equation*}
	\left\{
 \begin{aligned}
		&\frac{dx_i(t)}{dt}=v_i,\\
		&\frac{dv_i(t)}{dt} =-\lambda v_{i}-\frac{\kappa }{N}\sum_{j\neq i}\partial_{x_i}K\big(x_i(t)-x_j(t)\big),\quad i=1,\cdots,N ,\quad t>0,
  \end{aligned}
\right.
\end{equation*}
where $x_i(t)$ and $v_i(t)$ denote the position and velocity of the $i$-th particle at time $t>0$.
The case $\kappa >0$ and $\kappa <0$ correspond to repulsive and attractive potentials, respectively.
The interaction potential $K$ is given by the Riesz kernel:
\begin{align*}
	K(x)=|x|^{-\alpha},\quad d-2<\alpha<d.
\end{align*}


\vspace{2mm}

The compressible Euler equations with non-local forces have been investigated extensively with many significant results.  Without the damping term $\lambda \rho u$ and in the case of the Poisson coupling (i.e. $\alpha=d-2$), the global dynamics of solutions has been studied in many different settings. For instance, Guo \cite{Guo1} constructed global radially symmetric smooth  solutions in 3D for small altitude data with irrotational velocity due to the dispersive effect of the electric field (see also \cite{IP} for the 2D case).  Tadmor and Wei \cite{TW1} obtained the global regularity of solutions in 1D for a class of large data. As for the  compressible Euler equations with Riesz interactions, Choi \cite{choi} proved the local well-posedness  of classical solutions and  exhibited sufficient conditions of finite time blow-up phenomena. We also mention that  Danchin and Ducomet \cite{DD1,DD2,DD3}  dealt with the well-posedness issue in the case that the density contains vacuum and that the initial velocity allows some reference vector field. For the one-dimensional pressureless Euler flows with non-local forces, there exist critical thresholds between the subcritical region with global regularity and the supercritical region with finite-time blow-up of classical solutions, cf. \cite{CCTT,CCZ,ELT}.

The study of the system \eqref{E-R} is strongly related to the mathematical theory of the compressible Euler equations with damping:
\begin{equation}\label{euler}
\left\{
\begin{aligned}
& \partial_{t}\rho+\div (\rho u)=0,\\
& \partial_{t}(\rho u)+\div (\rho u\otimes u)+\nabla P(\rho)+\lambda\rho u=0.
\end{aligned}
\right.
\end{equation}
Indeed, the system  \eqref{E-R} with $\alpha=d$ can be viewed as the classical damped Euler system \eqref{euler} with the special pressure $P(\rho)=\frac{\kappa}{2}\rho^2$. Without the damping term $\lambda \rho u$ in \eqref{euler},  it is well-known that classical solutions exist locally in time and may develop the singularity (for example, shock wave) in finite time (e.g., refer to  \cite{dafermos1,majda1}).  On the other hand, it was observed in \cite{stw2003, wy2001} that the term $\lambda \rho u$  can prevent the formation of singularities and guarantee the global-in-time existence of classical solutions for \eqref{euler} with small initial perturbations in Sobolev spaces $H^s(\mathbb{R}^d)$ with $s>d/2+1$.  As pointed out by the theory of Shizuta and Kawashima (refer to \cite{Kawashima,kawashima-yong-2004,kawashima-yong-2009,SK,UKS} and references therein), the eigenvalues of the linearized system for the damped Euler equations \eqref{euler} satisfy 
\begin{align}\label{standard}
{\rm Re}~\lambda(\xi) \leq -\frac{c|\xi|^{2}}{1+|\xi|^{^{2}}},
\end{align}
which falls into  the class of partially dissipative hyperbolic systems with the standard dissipative structure.
Based on the Shizuta-Kawashima condition, Xu and Kawashima \cite{xu-kawashima-2014,xu2,xu3} investigated the global existence and optimal time-decay rates of solutions for (\ref{euler}) in the critical inhomogeneous Besov space $B^{d/2+1}_{2,1}$. By using the theory due to Shizuta-Kawashima \cite{Kawashima,SK,UKS} and Beauchard-Zuazua \cite{BZ1}, Crin-Barat and Danchin \cite{Barat-Danchin,Barat-Danchin2022} developed a functional setting to establish the global existence and large time behavior of solutions for initial perturbations in the critical hybrid Besov space where the low frequencies are only bounded in $\dot{B}^{d/2}_{2,1}$ instead of $L^2$. The authors \cite{Barat-Danchin,Barat-Danchin2023} also  investigated the global existence in functional spaces where the low frequencies are bounded in more general $L^p$-type spaces with $p\geq2$ (see also \cite{xu-zhang} in the case of general pressure laws). A complete review can be found in \cite{danchinnote}. In addition, there are many significant investigations concerning relaxation limits of \eqref{euler} toward the porous medium system, cf. \cite{Barat-Danchin2023,CS,danchinnote,junca1,mar0,xu00} and references therein.


It should be noted that without the pressure term, the multi-dimensional damped Euler system \eqref{euler} would not have enough dissipation of $\rho$ to ensure the global existence due to the fact that it violates the Shizhuta-Kawashima condition. However, when the effects of non-local forces are taken into account, the Euler equations with non-local forces may possess enough dissipation and admit global-in-time solutions. Danchin and Mucha \cite{danchinmucha23}   investigated the damped Euler system with a fuzzy nonlocal force as the approximation of the classical pressure law $P(\rho)=\frac{\kappa}{2}\rho^2$, and established the global existence and relaxation limit of solutions in the critical $L^2$ framework. For the pressureless damped Euler-Riesz system \eqref{E-R}, 
Choi and Jung \cite{choi and jung} investigated the global existence of solutions when the initial data is a small perturbation of the equilibrium state $(\bar{\rho},0)$ with $\bar{\rho}>0$ in  $H^m(\Omega)\times H^{m+(d-\alpha)/2}(\Omega)$ ($m>d/2+2$), where $\Omega$ is either $\mathbb{T}^{d}$ or $\mathbb{R}^{d}$ . Furthermore, they \cite{choi and jung} showed that as the time evolves, the solutions converge to $(\bar{\rho},0)$ in $H^m(\Omega)\times H^{m+(d-\alpha)/2}(\Omega)$ at exponential or algebraic rates under the additional regularity assumptions of  Sobolev spaces with negative indexes. On the other hand, as $\lambda\rightarrow\infty$, one can expect that under a suitable scaling, the system \eqref{E-R} converge to  the fractional porous medium flow (cf.\cite{fpm-caffa-1,fpm-caffa-2,fpm-caffa-decay}):
\begin{align}\label{fpm}
	&\partial_t\rho 
	+\kappa \div(\rho \nabla\Lambda^{\alpha-d}\rho )=0.
\end{align}
By means of relative entropy estimates, Choi and Jeong \cite{choi-jeong-2} provided a rigorous justification of the relaxation limit from the scaled system \eqref{E-R} to the fractional porous medium flow \eqref{fpm} in a finite time-interval.



\vspace{2mm}

In the paper, we aim to develop a functional setting to study the global well-posedness and long-time behavior of solutions for the pressureless Euler-Riesz system \eqref{E-R} with initial data around the constant state $(\bar{\rho},0)$. Compared with the recent advances by Choi and Jung \cite{choi and jung}, our results contain a larger class of initial data with critical regularity. As pointed out by many works 
\cite{Barat-Danchin2023,dafermos1,danchinnote,majda1,xu-kawashima-2014} for general hyperbolic systems, the Lipschitz bound of solutions may be the minimal regularity to ensure the uniqueness and determine the blow-up mechanism. Thus, the high-frequency $\dot{B}^{d/2+1}_{2,1}$ assumption on the initial data is critical due to the end-point embedding $ \smash{\dot{B}^{d/2+1}_{2,1} \hookrightarrow \dot{W}^{1,\infty}}$. Moreover, in the low-frequency regime, according to the spectral behaviors on the linear analysis, we are able to consider more general spaces $\dot{B}^{s}_{p,1}$, where the restrictions $p\geq2$ and $s\leq d/p$ ensure the control of nonlinear terms and stay in Banach spaces.


Another goal of this paper is to reveal the effects of fractional diffusion caused by Riesz interactions on  regularity estimates and decay rates of solutions for \eqref{E-R}. Our approach is inspired by {\emph{hypercoercivity}} arguments developed in the recent works  \cite{Barat-Danchin,Barat-Danchin2022,Barat-Danchin2023,CS}. One of the main difficulties and interests is that, the pressureless damped Euler-Riesz system has the following dissipative property:
\begin{align}\label{fract}
{\rm Re}~\lambda(\xi)\leq -\frac{c|\xi|^{\alpha-d+2}}{1+|\xi|^{\alpha-d+2}},
\end{align}
where $\lambda(\xi)$ denotes the eigenvalues of the linearized system for \eqref{E-R}.
Despite the fact that the system \eqref{E-R} is not first-order hyperbolic due to the Riesz interactions, it is observed that the solutions exhibit a damping behavior for high frequencies $|\xi|\gtrsim 1$. However, as for low frequencies $|\xi|\lesssim1$, such dissipative rate is {\emph{not standard}} and weaker than from the standard one \eqref{standard} due to ${\rm Re}~\lambda(\xi)\lesssim -|\xi|^{\alpha-d+2}$ with $\alpha-d+2\in(0,2)$. Therefore, it is natural and interesting to extend the previous methods  concerning the dissipative structure of standard type (e.g., \cite{Barat-Danchin2023}) to the non-standard one involving fractional diffusion.

\subsection{Reformulation and spectral behavior}\label{subsection12}

Without loss of generality, we set $\lambda=\kappa =\bar{\rho}=1$ throughout the rest of this paper. Defining the fluctuation $a=\rho-1$, we reformulate  the Cauchy problem \eqref{E-R}-\eqref{d} as follows
\begin{equation}\label{E-Rli}
	\left\{\begin{array}{l}
		\partial_ta+\div u =-\div(a u),\\
		\partial_t u+ u+\nabla\Lambda^{\alpha-d}a
		=-u\cdot\nabla u,\\
       (a,u)(0,x)=(a_0,u_0)(x)
	\end{array}\right.
\end{equation}
with $\alpha\in(d-2,d)$ and $a_0\triangleq\rho_0-1.$

In order to understand the behavior of the solution to 
\eqref{E-Rli}, we analyse the eigenvalues of the linearized system. The linearization of \eqref{E-Rli} gives
\begin{equation}\label{E-R-lin}
	\left\{\begin{array}{l}
		\partial_ta+\div u =0,\\
		\partial_t u+ u+\nabla\Lambda^{\alpha-d}a
		=0.
	\end{array}\right.
\end{equation}
By the Hodge decomposition
$$\omega\triangleq \Lambda^{-1}\curl u,\quad\quad m\triangleq \Lambda^{-1}\div u,$$
the system \eqref{E-R-lin} can be rewritten as the $2\times 2$ coupled system
\begin{equation}\label{E-R-lin-hodge}
	\left\{\begin{array}{l}
		\partial_ta+\Lambda m=0,\\
		\partial_t m+  m-\Lambda^{\alpha-d+1}a
		=0,
	\end{array}\right.
\end{equation}
and the purely damped equation
$$\partial_t\omega+\omega=0.
$$
Employing the Fourier transform for \eqref{E-R-lin-hodge}, we have
\begin{equation*}
	\frac{d}{dt}\left(\begin{array}{l}
		\hat{a}\\
		\hat{m}
	\end{array}
	\right)
	=A(\xi)
	\left(\begin{array}{l}
		\hat{a}\\
		\hat{m}
	\end{array}
	\right),\quad\quad A(\xi)\triangleq\left(\begin{matrix}
	0	&  -|\xi|\\
	|\xi|^{\alpha-d+1}	& -1
\end{matrix}\right).
\end{equation*}
The eigenvalues of the matrix $ A(\xi)$ can be computed as
$$
\lambda_{1,2}=-\frac{1}{2}\pm
\frac{\sqrt{1-4|\xi|^{\alpha-d+2}}}{2}.$$
 Therefore, one can see that $\lambda_{1,2}$ satisfy the property \eqref{fract}. More precisely, 
\begin{itemize}
\item In the low-frequency regime $|\xi|\lesssim 1$, the eigenvalues $\lambda_{1}$ and $\lambda_{2}$ are both  real and asymptotically equivalent to $-|\xi|^{\alpha-d+2}$ and $-1$, respectively. This implies that the fractional diffusion effect and the damping effect coexist.

\item  In the high-frequency regime $|\xi|\gtrsim 1$, the eigenvalues $\lambda_{1}$ and $\lambda_{2}$ are complex conjugates and asymptotically approach
$-\frac12\pm i|\xi|^{\frac{\alpha-d+2}{2}}$ respectively.
\end{itemize}
The above spectral behaviors reveal the sharp dissipative structure of solutions in two different frequency regimes. Since the high frequencies decay exponentially, one expects that the optimal decay rates of solutions are exactly the same as those of the fractional heat equations. In addition, we observe that the real eigenvalues avoid dispersive effects in low frequencies, so it is possible to establish the global existence in general $L^p$-type spaces, where $p$ is not just $2$. These behaviors in different frequencies motivate us to employ a functional framework based on Littlewood-Paley decomposition and Besov spaces so as to reflect maximal regularity properties of solutions to the nonlinear problem \eqref{E-Rli}.

\subsection{Notations and functional spaces}
Before stating our main results, we list some notations that are used frequently throughout the paper.
For simplicity, $C$ denotes some generic positive constant. $A\lesssim B$ ($A\gtrsim B$) means that $A\leq C B$ ($A\geq C B$),  while $A\sim B$ means that both $A\lesssim  B$
and $A\gtrsim B.$ 
For $A=(A_{i,j})_{1\leq i,j\leq d}$ and $B=(B_{i,j})_{1\leq i,j\leq d}$ two $d\times d$ matrices, we denote $A:B=\Tr AB=\sum_{i,j}A_{i,j}B_{j,i}.$
For a Banach space $X$, $p\in[1, \infty]$ and
$T>0$, the notation $L^p(0, T; X)$ or $L^p_T(X)$ designates the set of measurable functions $f: [0, T]\to X$ with $t\mapsto\|f(t)\|_X$ in $L^p(0, T)$, endowed with the norm $\|\cdot\|_{L^p_{T}(X)} \triangleq\|\|\cdot\|_X\|_{L^p(0, T)}$, and $\mathcal{C}([0,T];X)$ denotes the set of continuous functions $f: [0, T]\to X$. Let $\mathcal{F}(f)=\widehat{f}$ and $\mathcal{F}^{-1}(f)=\breve{f}$ be the Fourier transform of $f$ and its inverse.


Then, we recall the Littlewood-Paley decomposition and the definitions of Besov spaces. The reader can refer to Chapters 2 and 3 in \cite{BCD} for more details. 
Choose a smooth radial non-increasing function $\chi(\xi)$  compactly supported in $B(0,\frac{4}{3})$ and satisfying $\chi(\xi)=1$ in $B(0,\frac{3}{4})$. Then, $\varphi(\xi)\triangleq\chi(\xi/2)-\chi(\xi)$ satisfies
$$
\sum_{j\in \mathbb{Z}}\varphi(2^{-j}\cdot)=1,\quad \text{{\rm{Supp}}}~ \varphi\subset \Big{\{}\xi\in \mathbb{R}^{d}~|~\frac{3}{4}\leq |\xi|\leq \frac{8}{3}\Big{\}}.
$$
For any $j\in \mathbb{Z}$, define the homogeneous dyadic block
$$
u_j=\dot{\Delta}_{j}u\triangleq\mathcal{F}^{-1}\big{(} \varphi(2^{-j}\cdot )\mathcal{F}(u) \big{)}=2^{jd}h(2^{j}\cdot)\star u,\quad\quad h\triangleq\mathcal{F}^{-1}\varphi.
$$
We also define the low-frequency cut-off operator
\begin{align}
	\dot{S}_{j}\triangleq\sum_{j'\leq j-1}\dot{\Delta}_{j'}.\nonumber
\end{align}
Let $\mathcal{S}_{h}'$ stand for the set of tempered distributions $z$ on $\mathbb{R}^{d}$ such that $\dot{S}_{j}z\rightarrow0$ uniformly as $j\rightarrow\infty$ (i.e., modulo polynomials). Then $u$ has the decomposition 
\begin{equation}\nonumber
	\begin{aligned}
		&u=\sum_{j\in \mathbb{Z}}\dot{\Delta}_{j}u\quad\text{in}~\mathcal{S}',\quad \forall u\in \mathcal{S}_{h}' \quad \text{with} \quad \dot{\Delta}_{j}\dot{\Delta}_{l}u=0,\quad\text{if}\quad|j-l|\geq2.
	\end{aligned}
\end{equation}

Due to those dyadic blocks, we give the definitions of homogeneous Besov spaces and mixed space-time Besov spaces as follow. For $s\in \mathbb{R}$ and $1\leq p,r\leq \infty$, the  homogeneous Besov space $\dot{B}^{s}_{p,r}$ is defined by
$$
\dot{B}^{s}_{p,r}\triangleq\big{\{} u\in \mathcal{S}_{h}'~:~\|u\|_{\dot{B}^{s}_{p,r}}\triangleq\|\{2^{js}\|\dot{\Delta}_{j}u\|_{L^{p}}\}_{j\in\mathbb{Z}}\|_{l^{r}}<\infty \big{\}} .
$$
When $p=r=2$, the space $\dot{B}^{s}_{2,2}$ is equivalent to the homogeneous Sobolev space $\dot{H}^{s}(\mathbb{R}^d)$.

Furthermore, we recall a class of mixed space-time Besov spaces $\widetilde{L}^{\varrho}(0,T;\dot{B}^{s}_{p,r})$ introduced by Chemin-Lerner \cite{chemin-lerner}:
$$
\widetilde{L}^{\varrho}(0,T;\dot{B}^{s}_{p,r})\triangleq \big{\{} u\in L^{\varrho}(0,T;\mathcal{S}'_{h})~:~ \|u\|_{\widetilde{L}^{\varrho}_{T}(\dot{B}^{s}_{p,r})}\triangleq\|\{2^{js}\|\dot{\Delta}_{j}u\|_{L^{\varrho}_{T}(L^{p})}\}_{j\in\mathbb{Z}}\|_{l^{r}}<\infty
\big{\}}
$$
for $T>0$, $s\in\mathbb{R}$ and $1\leq \varrho,r \leq \infty$. By Minkowski's inequality, it holds
\begin{equation}\nonumber
\begin{aligned}
&\|u\|_{\widetilde{L}^{\varrho}_{T}(\dot{B}^{s}_{p,r})}\leq(\geq) \|u\|_{L^{\varrho}_{T}(\dot{B}^{s}_{p,r})},\quad\text{if}~r\geq(\leq)\varrho,
\end{aligned}
\end{equation}
where $\|\cdot\|_{L^{\varrho}_{T}(\dot{B}^{s}_{p,r})}$ is the usual Lebesgue-Besov norm. 

Let the threshold  $J_{1}$ between low and high frequencies be given by \eqref{J_1}. In order to restrict Besov norms to the low frequency part and the high-frequency part, we write $\|\cdot\|_{\dot B_{q_1,r}^{s_1}}^{\ell}$ and $\|\cdot\|_{\dot B_{q_2,r}^{s_2}}^{h}$
 to denote Besov semi-norms, that is,
\begin{equation*}\label{DefHLB}
	\|u\|_{\dot B_{q_1,r}^{s_1}}^{\ell}\triangleq \Big(\sum_{j\leq J_{1}}
	\big(2^{s_1j}\|\dot \Delta_j u\|_{L^{q_1}}\big)^r\Big)^{\frac{1}{r}}
	\quad \mbox{and}\quad
	\|u\|_{\dot B_{q_2,r}^{s_2}}^{h}\triangleq \Big(\sum_{j\geq J_{1}-1}
	\big(2^{s_2j}\|\dot \Delta_j u\|_{L^{q_2}}\big)^r\Big)^{\frac{1}{r}}.
\end{equation*}
One can deduce that for all $\sigma_0>0$,
\begin{equation}\label{HLEst}
	\|u\|_{\dot B_{q_1,r}^{s_1}}^{\ell}
	\leq
	2^{\sigma_0J_{1}}\|u\|_{\dot B_{q_1,r}^{s_1-\sigma_0}}^{\ell},
    \quad
	\|u\|_{\dot B_{q_2,r}^{s_2}}^{h}
	\leq
	2^{-\sigma_0J_{1}+\sigma_0}\|u\|_{\dot B_{q_2,r}^{s_2+\sigma_0}}^{h}
\end{equation}
and
\begin{equation}\label{low-1-infty}
\|u\|_{\dot B_{q_1,1}^{s}}^{\ell}
	\leq
	\sum_{j\leq J_{1}}2^{j\sigma_0}\|u\|_{\dot B_{q_1,\infty}^{s-\sigma_0}}^{\ell}
	\lesssim
	2^{J_{1}\sigma_0}\|u\|_{\dot B_{q_1,\infty}^{s-\sigma_0}}^{\ell}.
\end{equation}
Denote by $\dot{B}^{s_1,s_2}_{p,2}$  the hybrid space in $\dot{S}'_{h}$ endowed with the norm 
$$
\|\cdot\|_{\dot{B}^{s_1}_{p,1}}^{\ell}+\|\cdot\|_{\dot{B}^{s_2}_{2,1}}^{h}.
$$
We also introduce the low-high-frequency decomposition $u=u^{\ell}+u^{h}$ with
\begin{equation}\nonumber
	u^\ell\triangleq\sum_{j\leq J_1-1} \dot\Delta_j u=\dot{S}_{J_1}u
	\quad\mbox{and}\quad
	u^h\triangleq\sum_{j\geq J_1} \dot\Delta_j u=( {\rm{Id}}-\dot{S}_{J_1})u.
\end{equation} 
It is easy to check that
\begin{equation*}
 \|u^{\ell}\|_{\dot B_{q_1,r}^{s}}
	\leq
	\|u\|_{\dot B_{q_1,r}^{s}}^{\ell}
	\quad\mbox{and}\quad
	\|u^{h}\|_{\dot B_{q_2,r}^{s}}
	\leq
	\|u\|_{\dot B_{q_2,r}^{s}}^{h}.
\end{equation*}
\medbreak

\subsection{Main results}

We now state our first  results as follows about the global existence and uniqueness of solutions to the Cauchy problem \eqref{E-Rli}. 

\begin{theorem}\label{global}
Let $d\geq1$, $d-2<\alpha<d$, $s_*\triangleq\frac{\alpha-d+2}{2}\in (0,1)$
and
\begin{equation}\label{p}
\left\{\begin{array}{l}
2\leq p\leq4,\quad\quad \quad  \text{if}\quad  1\leq d\leq4,\\
2\leq p\leq\frac{2d}{d-2},\quad \quad \text{if}\quad d\geq 5.
\end{array}\right.
\end{equation}
There is a constant $\delta_0>0$ depending only  on $d$, $p$ and $ \alpha$ such that if the initial datum $(a_0,u_0)$ fulfills
$a_0\in\dot B^{\frac dp-1,\frac{d}{2}+1}_{p,2}$, $u_0\in
\dot B^{\frac dp,\frac{d}{2}+2-s_*}_{p,2}$ and
\begin{equation}\label{x-p-0}
\cE_{p,0}\triangleq\|a_0\|_{
\dot B^{\frac dp-1}_{p,1}}^{\ell}+\|u_0\|_{\dot B^{\frac dp}_{p,1}}^{\ell}
+\|a_0\|_{\dot B^{\frac d2+1}_{2,1}}^{h}
+\|u_0\|_{\dot B^{\frac{d}{2}+2-s_*}_{2,1}}^{h}
\leq \delta_0,
\end{equation}
then the Cauchy problem \eqref{E-Rli} admits a unique global strong solution $(a,u)$ satisfying 
\begin{equation}\label{global-Var0}
\begin{aligned}
&\|a\|_{\tilde L_t^\infty(\dot B^{\frac dp-1}_{p,1})
}^{\ell}
+\|u\|_{\tilde L_t^\infty(\dot B^\frac dp_{p,1})}^{\ell}
+\|a\|_{\tilde L_t^\infty(\dot B^{\frac d2+1}_{2,1})}^{h}
+\|u\|_{\tilde L_t^\infty(\dot B^{\frac{d}{2}+2-s_*}_{2,1})}^{h}\\
&\quad+\|a\|_{L_t^1(\dot B^{\frac dp-1+2s_*}_{p,1})}^{\ell}
+\|u\|_{L_t^1(\dot B^{\frac dp}_{p,1})}^{\ell}\\
&\quad+\|a\|_{L_t^1(\dot B^{\frac d2+1}_{2,1})}^{h}
+\|u\|_{L_t^1(\dot B^{\frac{d}{2}+2-s_*}_{2,1})}^{h}\leq C\cE_{p,0},\quad\quad t>0,
\end{aligned}
\end{equation}	
 where $C>0$ is a universal constant.
\end{theorem}

\begin{remark}
Theorem \ref{global} gives an improvement of the global regularity for the Cauchy problem of the pressureless damped Euler-Riesz system in comparison with the recent effort \cite{choi and jung}. For any $m>\frac{d}{2}+2$, we have the embeddings 
$$H^{m}
\hookrightarrow \dot{B}^{\frac{d}{p}-1,\frac{d}{2}+1}_{p,2}(2\leq p\leq d),
\quad\quad H^{m+\frac{d-\alpha}{2}}\hookrightarrow \dot{B}^{\frac{d}{p},\frac{d}{2}+2-s_*}_{p,2}(p\geq 2).
$$
Furthermore, it is shown by \eqref{global-Var0} that the low frequencies of $a$ behave like fractional heat kernel,  while $u$ and the high frequencies  exhibit damping effects. Such qualitative estimates  are sharp and coincide with the spectral behaviors of solutions  presented in Subsection \ref{subsection12}.
\end{remark}


\begin{remark}
It should be pointed out that depending on $s_*$, $a$ and $u$ exhibit different behaviors in low frequencies due to the coupling $u+\nabla\Lambda^{\alpha-d}a$ in the velocity equations. When $\frac{1}{2}<s_*<1$, i.e., $d-1<\alpha<d$, the regularity $L^1_t(\dot{B}^{d/p}_{p,1})$ of $u$ is stronger than the regularity $L^1_t(\dot{B}^{d/p-1+2s_*}_{p,1})$ of $a$, so the damping effect is dominant. On the other hand, in the case $0<s_*<\frac{1}{2}$, i.e., $d-2<\alpha<d-1$, the fractional diffusion effect caused by Riesz interactions is stronger in the low-frequency regime. Hence, $s_*=\frac{1}{2}$ is the threshold between damping and fractional diffusion effects. This phenomenon also occurs in the decay rates of $a$ and $u$ obtained in Theorem \ref{decay-a-u} below.  
\end{remark}


Next, we deduce the large-time behavior of  solutions constructed in Theorem \ref{global}.

\begin{theorem}\label{decay-a-u}
Let $d\geq1$, $d-2<\alpha<d$, $s_*\triangleq\frac{\alpha-d+2}{2}\in (0,1)$, and $p$ be given by \eqref{p}. Let $(a,u)$ be the global
solution to the Cauchy problem \eqref{E-Rli} addressed in Theorem \ref{global}. If in addition to \eqref{x-p-0}, suppose further 
$a_0^\ell\in \dot B^{\sigma_1}_{p,\infty}$ and $u_0^\ell\in
\dot B^{\sigma_1+1}_{p,\infty}$ with $-\frac dp-1\leq\sigma_1<\frac dp-1$ 
such that 
\begin{equation}\label{Xp0}
\cX_{p,0}\triangleq\|a_0^\ell\|
_{\dot B^{\sigma_1}_{p,\infty}}
+\|u_0^\ell\|
_{\dot B^{\sigma_1+1}_{p,\infty}}
+\|a_0\|_{\dot{B}^{\frac{d}{2}+1}_{2,1}}^{h}
+\|u_0\|_{\dot B^{\frac{d}{2}+2-s_*}_{2,1}}^{h}
\end{equation}
is bounded, then it holds for all $t\geq0$ that
\begin{equation}\label{decayau}
\begin{aligned}
\|a\|_{\dot{B}_{p,1}^\sigma}&
\lesssim 
\begin{cases}
 \cX_{p,0}(1+t)^{-\frac{1}{2s_*}(\sigma-\sigma_1) },
& \mbox{\quad if \quad $\sigma_1<\sigma\leq\frac d{p}-1$},\\
 \cX_{p,0}(1+t)^{-\frac{1}{2s_*}(\frac{d}{p}-1-\sigma_1) },
& \mbox{\quad if \quad $\frac d{p}-1<\sigma\leq \frac{d}{p}+1$},
\end{cases}
\end{aligned}
\end{equation}
and
\begin{equation}\label{decayau1}
\begin{aligned}
\|u\|_{\dot{B}_{p,1}^{\sigma}}&\lesssim  
\begin{cases}
 \cX_{p,0}(1+t)^{-\frac{1}{2s_*}(\sigma-\sigma_1+2s_*-1) },
& \mbox{\quad if \quad $\sigma_1+1<\sigma\leq\frac d{p}-2s_*$},\\
 \cX_{p,0}(1+t)^{-\frac{1}{2s_*}(\frac{d}{p}-1-\sigma_1) },
& \mbox{\quad if \quad $\frac{d}{p}-2s_*<\sigma\leq \frac{d}{p}+2-s_*$},
\end{cases}
\end{aligned}
\end{equation}
and the high frequency norms satisfy
\begin{equation}\label{au-high-decay}
\|a\|_{\dot{B}_{2,1}^{\frac d2+1}}^h
+\|u\|_{\dot{B}_{2,1}^{\frac d2+2-s_*}}^h
\lesssim \cX_{p,0}
(1+t)^{-\frac{1}{s_*}(\frac dp-1-\sigma_1)}.
\end{equation}
\end{theorem}

\begin{remark} The inequalities 
\eqref{decayau}-\eqref{decayau1} exhibit the optimal time-decay rates  for the fractional heat equation {\rm(}\cite{BS,brandolese1}{\rm)}, which are slower than that of partially dissipative hyperbolic systems with the standard dissipative structure {\rm(}\cite{xu2,xu3}{\rm)}. Indeed, according to the expansions of the eigenvalues for the linear system {\rm(}see Subsection \ref{subsection12}{\rm)}, the solution behaves like the fractional heat kernel at low frequencies while the high-frequency part decays at a faster rate. On the other hand, since the relaxation approximation is given by the fractional porous media model \eqref{fpm}, one can expect that the solution is asymptotically equivalent to that of \eqref{fpm}.  
\end{remark}

\begin{remark}
The regularity condition $\dot{B}^{\sigma_1}_{p,\infty}$ is sharp for the optimal decay estimates of dissipative systems {\rm(}\cite{brandolese1}{\rm)}. The  space $\dot{B}^{\sigma_1}_{p,\infty}$ is less restrictive than the standard assumption with respect to $L^1$ or $\dot{H}^{-s}$, if we keep in mind observing  the following embeddings
$$
L^{\frac{p}{2}}\hookrightarrow \dot{B}^{-\frac{d}{p}}_{p,\infty}~(p\geq2),\quad\quad \dot{H}^{-s}\hookrightarrow \dot{B}^{-s}_{2,\infty}.
$$
\end{remark}

As a direct consequence of Theorems \ref{global}-\ref{decay-a-u}, one can also get the global existence and optimal decay estimates of solutions in the classical $L^2$-type critical spaces.

\begin{cor}
Assume $d\geq1$, $d-2<\alpha<d$ and $s_*\triangleq\frac{\alpha-d+2}{2}\in (0,1)$.  There is a constant $\delta^*_0>0$ depending only  on $d$ and $ \alpha$ such that if the initial datum $(a_0,u_0)$ fulfills
$a_0\in\dot B^{\frac d2-1}_{2,1}\cap \dot B^{\frac{d}{2}+1}_{2,1}$, $u_0\in
\dot B^{\frac d2}_{2,1}\cap \dot{B}^{\frac{d}{2}+2-s_*}_{2,1}$ and
\begin{equation}\nonumber
\|a_0\|_{\dot B^{\frac d2-1}_{2,1}\cap \dot B^{\frac{d}{2}+1}_{2,1}}+\|u_0\|_{\dot B^{\frac d2}_{2,1}\cap \dot{B}^{\frac{d}{2}+2-s_*}_{2,1}}\leq \delta^*_0,
\end{equation}
then the Cauchy problem \eqref{E-Rli} admits the unique  global solution $(a,u)$   satisfying
\begin{equation}\nonumber
\left\{
\begin{aligned}
&a^{\ell}\in \tilde{L}^\infty(\mathbb{R}_{+};\dot B^{\frac d2-1}_{2,1})
\cap L^1(\mathbb{R}_{+};\dot B^{\frac d2-1+2s_*}_{2,1}),\\
&a^{h}\in \tilde L^\infty(\mathbb{R}_{+};\dot B^{\frac d2+1}_{2,1})
\cap L^1(\mathbb{R}_{+};\dot B^{\frac d2+1}_{2,1}),\\
&u^{\ell}\in \tilde L^\infty(\mathbb{R}_{+};\dot B^\frac d2_{2,1})
\cap L^1(\mathbb{R}_{+};\dot B^{\frac d2}_{2,1}),\\
&u^{h}\in \tilde L^\infty(\mathbb{R}_{+};\dot B^{\frac{d}{2}+2-s_*}_{2,1})
\cap L^1(\mathbb{R}_{+};\dot B^{\frac{d}{2}+2-s_*}_{2,1}).
\end{aligned}
\right.
\end{equation}	

Additionally, if $a_0^\ell\in \dot B^{\sigma_1}_{2,\infty}$ and $u_0^\ell\in
\dot B^{\sigma_1+1}_{2,\infty}$ with $-\frac d2-1\leq\sigma_1<\frac d2-1$  then for all $t\geq 0$, we have
\begin{equation}\nonumber
\begin{aligned}
&\|\Lambda^{\sigma} a\|_{L^2}\lesssim 
\begin{cases} 
 (1+t)^{-\frac{1}{2s_*}(\sigma-\sigma_1)},
 & \mbox{\quad if \quad $ \sigma_1<\sigma\leq \frac{d}{2}-1$},\\
 (1+t)^{-\frac{1}{2s_*}(\frac{d}{2}-1-\sigma_1)},
 & \mbox{\quad if \quad $ \frac{d}{2}-1<\sigma\leq \frac{d}{2}+1$},
\end{cases}
\end{aligned}
\end{equation}
and
\begin{equation}\nonumber
\begin{aligned}
&\|\Lambda^{\sigma} u\|_{L^2}\lesssim
\begin{cases} 
(1+t)^{-\frac{1}{2s_*}(\sigma-\sigma_1+2s_*-1)},
& \mbox{\quad if \quad $ \sigma_1+1<\sigma\leq \frac{d}{2}-2s_*$},\\
(1+t)^{-\frac{1}{2s_*}(\frac{d}{2}-1-\sigma_1)},& \mbox{\quad if \quad $ \frac{d}{2}-2s_*<\sigma\leq \frac{d}{2}+2-s_*$}.
\end{cases}
\end{aligned}
\end{equation}
\end{cor}

\vspace{1mm}

\subsection{Strategies} 

We explain the main strategies to prove Theorems \ref{global}-\ref{decay-a-u} concerning the study of the Euler-Riesz system with critical regularity. 
Different from the previous efforts \cite{Barat-Danchin,Barat-Danchin2022,Barat-Danchin2023} on the first-order hyperbolic system, the hyperbolic part of \eqref{E-R-lin} only provides the dissipation of $u$. To overcome the difficulty, we need to capture the intrinsic dissipation of $a$ by the coupling of the Riesz term and the damping part, which is weaker than the standard dissipation for first-order hyperbolic systems. The new structure in the Euler-Riesz system relies on the elaborate low-frequency and high-frequency analysis via the Littlewood-Paley decomposition.

In the low-frequency regime, since the usual symmetrization cannot be applied to the $L^p$ setting, we perform an elaborate energy argument. Precisely, we decompose the system \eqref{E-R-lin} into a fractional diffusion equation and a damping equation due to the introduction the {\emph{effective unknown}}
$$z\triangleq u
+\nabla\Lambda^{\alpha-d}a,$$
which is inspired by \cite{Barat-Danchin2022,Barat-Danchin2023,CS} in the study of one-order partially dissipative hyperbolic systems. Here we replace the usual pressure term by the non-local one $\nabla\Lambda^{\alpha-d}a$. In fact, the effective unknown $z$ allows us to rewrite the equation $\eqref{E-Rli}_1$ as
\begin{equation}
\label{low-a0}
\begin{array}{l}
\partial_ta
+\Lambda^{2s_*}a
=-\div z -\div(a u).
\end{array}
\end{equation}
This exhibits the fractional diffusion term $\Lambda^{2s_*}a$ with $0<s_*<1$. In order to analyze the linear term $-\div z$ on the right-hand side of \eqref{low-a0}, we observe that $z$ satisfies the damped equation
\begin{align}\label{low-z0}
\partial_t z+ z=-\nabla\Lambda^{2s_*-2}\div z
-\nabla\Lambda^{4s_*-2}a-\nabla\Lambda^{2s_*-2}\div(au)
-u\cdot\nabla u.
\end{align}
This gives a damping property of $z$ in low frequencies. To decouple $a$ and $z$, one can expect that the higher order terms on the right-hand sides of \eqref{low-a0} and \eqref{low-z0} can be absorbed if the threshold $J_{1}$ between low and high frequencies is chosen to be suitable small. Due to the fact that the dissipation of \eqref{low-a0} is weaker than the standard heat equation, we perform the $\dot{B}^{d/p-1}_{p,1}$-estimate of $a$ and the $\dot{B}^{d/p}_{p,1}$-estimate of $u$, which are different from the damped compressible Euler equations (\cite{Barat-Danchin2022,Barat-Danchin2023}) where both $a$ and $u$ are analyzed at the $\dot{B}^{d/p}_{p,1}$ level. Thus, we are able to establish the maximal regularity estimates for $a$ and $z$, and then recover the desired estimates of $u$ (refer to Lemma \ref{lemma-low-global}).

In the high-frequency regime, we employ a hypocoercivity-type argument in the sense of localized frequencies. 
The major difficulty lies in the fact that the entropy in the study of first-order hyperbolic systems (\cite{kawashima-yong-2004,kawashima-yong-2009,yong}) may not be applied to symmetrize the system \eqref{E-R-lin-hodge} due to the fractional operator in the Riesz force. 
To overcome it, we observe that the linear coupled system \eqref{E-R-lin-hodge} in  terms of the unknowns $(U_1,U_2)=(\Lambda^{s_*}a,\Lambda m)$ can be rewritten as a symmetric system with relaxation:
\begin{equation}\nonumber
	\left\{\begin{array}{l}
		\partial_tU_1+\Lambda^{s_*} U_2=0,\\
		\partial_t U_2-\Lambda^{s_*}U_1+ U_2=0.
	\end{array}\right.
\end{equation}
This enables us to adapt the theory of symmetric hyperbolic systems to cancel the higher-order linear terms and capture the dissipation of $a$ in the spirit of hypocoercivity. For the nonlinear system \eqref{E-Rli}, one needs to overcome the higher order nonlinear terms $\div(au)$ and $u\cdot\nabla u$. By rewriting the system with some commutators for spectral localization and fractional Laplacian, we construct a delicate Lyapunov functional inequality to establish the desired estimates in the $L^2$ framework (cf. Lemma \ref{lemma-high-global}).

Finally, we generalize our recent Lyapunov energy argument from the classical heat-like dissipation to the fractional dissipation in the $L^p$ framework so as to deduce the optimal decay estimates of solutions for \eqref{E-Rli}. The crucial part of the proof of Theorem \ref{decay-a-u} is the evolution of the $\dot B_{p,\infty}^{\sigma_1}$-norm for low frequencies (see Proposition \ref{negative besov bounded} for details).

\vspace{4mm}

The rest of the paper unfolds as follows: 
In Section \ref{global-proof}, we establish the uniform a-priori estimate and give the proof of Theorem \ref{global}. In Section \ref{decay-proof}, we focus on large-time behaviors of solutions addressed by Theorem \ref{decay-a-u}.
 Some technical lemmas are recalled in Appendix.
 


\section{Proof of Theorem \ref{global}}\label{global-proof}
\subsection{A priori estimates}

Below, we give the key a-priori estimates of solutions, which leads to the global existence for $\eqref{E-Rli}$.
\begin{proposition}\label{priori-global}
Let $d\geq 1$, $p$ satisfy \eqref{p} and 
$s_*\triangleq\frac{\alpha-d+2}{2}\in (0,1)$.
Any given time $T>0,$ suppose that $(a,u)(t)$ with $0\leq t<T$ is a strong solution to the Cauchy problem \eqref{E-Rli}.
There exists a universal constant $\var_0>0$ such that 
\begin{equation}\label{a-priori-supose}
\|a\|_{L^\infty_t(L^\infty)}\leq \var_0,
\end{equation}
then it holds that
\begin{equation}\label{priori-estimate}
\cE_p(t)+\cD_p(t)\leq C_0\big(\cE_{p,0}+\cE_p(t)\cD_p(t)\big),
\end{equation}
where $C_0>0$ is a constant independent of $T$ and the functionals $\cE_p(t)$ and
$\cD_{p}(t)$ are, respectively, defined as 
\begin{equation*}
\begin{aligned}
\cE_p(t)&\triangleq
\|a\|_{\tilde L_t^\infty(\dot B^{\frac dp-1}_{p,1})}^{\ell}
+\|u\|_{\tilde L_t^\infty(\dot B^\frac dp_{p,1})}^{\ell}
+\|a\|_{\tilde L_t^\infty(\dot B^{\frac d2+1}_{2,1})}^h
+\|u\|_{\tilde L_t^\infty(\dot B^{\frac d2+2-s_*}_{2,1})}^h,
\end{aligned}
\end{equation*}	
and
\begin{equation}\label{cDpt}
\begin{aligned}
\cD_{p}(t)&\triangleq\|a\|_{\tilde L_t^1
	(\dot B^{\frac dp-1+2s_*}_{p,1})}^{\ell}
+\|u\|_{\tilde L_t^1
	(\dot B^{\frac dp}_{p,1})}^{\ell}
+\|a\|_{\tilde L_t^1
	(\dot B^{\frac d2+1}_{2,1})}^{h}
\\&\quad+
\|u\|_{\tilde L_t^1
	(\dot B^{\frac d2+2-s_*}_{2,1})}^{h}
+\|\partial_ta\|_{\tilde L_t^1(\dot B^{\frac dp}_{p,1})}.
\end{aligned}
\end{equation}
\end{proposition}

The proof of Proposition \ref{priori-global} is divided into the following two steps.
\subsubsection{Low-frequency analysis}

First, we establish e uniform estimates for low frequencies in the $L^{p}$ framework.
\begin{Lemma}\label{lemma-low-global}
Under the assumptions of Proposition \ref{priori-global}, it holds that
\begin{equation}\label{low-global}
\begin{aligned}
&\|a\|_{\tilde L_t^\infty(\dot B^{\frac dp-1}_{p,1})}^{\ell}
+\|u\|_{\tilde L_t^\infty(\dot B^\frac dp_{p,1})}^{\ell}
+\|a\|_{\tilde L_t^1
(\dot B^{\frac dp-1+2s_*}_{p,1})}^{\ell}
+\|u\|_{\tilde L_t^1
(\dot B^{\frac dp}_{p,1})}^{\ell}
+\|\partial_ta\|_{\tilde L_t^1(\dot B^{\frac dp}_{p,1})}^\ell
\\&\quad\lesssim\cE_{p,0}+\cE_p(t)\cD_p(t).
\end{aligned}
\end{equation}
\end{Lemma}
\begin{proof}
Defining the effective velocity 
$$z\triangleq u
+\nabla\Lambda^{\alpha-d}a,$$
we  rewrite the equation $\eqref{E-Rli}_1$ as
\begin{equation}
\label{low-a}
\begin{array}{l}
\partial_ta
+\Lambda^{2s_*}a
=-\div z -\div(a u).
\end{array}
\end{equation}
Applying $\dot\Delta_j$ to $\eqref{low-a}$,  multiplying both sides of it by $|a_j|^{p-2}a_j$ with $a_j=\dot\Delta_j a$, and then integrating the resulting equation over $\R^d$,
we derive
\begin{equation*}
\begin{aligned}
&\frac 1p\frac d{dt}\|a_j\|_{L^p}^p +2^{j2s_*}\|a_j\|_{L^p}^p
\\&\quad\lesssim\big(\|\div z_j\|_{L^p}
+\|\div (au)_j\|_{L^p}\big)
\|a_j\|_{L^p}^{p-1},
\end{aligned}
\end{equation*}
where we used Lemma \ref{low-fra-Lp}. Hence, Lemma \ref{difference-delta} ensures that
\begin{align*}
&\|a_j\|_{L^\infty_t(L^p)} +2^{j2s_*}\|a_j\|_{L^1_t(L^p)}
\\&\quad\lesssim\|(a_0)_j\|_{L^p} +\|\div z_j\|_{L^1_t(L^p)}
+\|\div (au)_j\|_{L^1_t(L^p)}.
\end{align*}
Multiplying both sides of the above inequality by $2^{j(\frac dp-1)}$, summing up over $j\leq J_1$ \rm(which will be chosen below\rm), 
and then using Bernstein's inequality,
we infer that
\begin{equation}\label{lowa}
\begin{aligned}
&\|a\|_{\tilde L_t^\infty
(\dot B^{\frac dp-1}_{p,1})}^{\ell}
+\|a\|_{\tilde L_t^1
(\dot B^{\frac dp-1+2s_*}_{p,1})}^{\ell}
\\&\quad\leq C_{1}\big(\|a_{0}\|_{
\dot B^{\frac dp-1}_{p,1}}^{\ell}
+\|z\|_{\tilde L_t^1
(\dot B^{\frac dp}_{p,1})}^{\ell}
+\|au\|_{\tilde L_t^1
(\dot B^{\frac dp}_{p,1})}^{\ell}\big).
\end{aligned}
\end{equation}

To handle the second term on the right-side of \eqref{lowa}, one needs to estimate the effective 
velocity $z.$ According to the definition of $z$ and $\eqref{E-Rli}$, the equation of $z$ reads 
\begin{align}\label{low-z}
\partial_t z+ z=-\nabla\Lambda^{2s_*-2}\div z
-\nabla\Lambda^{4s_*-2}a-\nabla\Lambda^{2s_*-2}\div(au)
-u\cdot\nabla u.
\end{align}
Applying $\dot\Delta_j$ to $\eqref{low-z}$, then multiplying both sides of it by $|z_j|^{p-2}z_j$ and integrating the resulting equation over $\R^d$,
we have
\begin{equation*}
\begin{aligned}
&\frac 1p\frac d{dt}\|z_j\|_{L^p}^p +\|z_j\|_{L^p}^p
\\&\quad\lesssim\big(\|\nabla\Lambda^{2s_*-2}\div z_j\|_{L^p}
+\|\nabla\Lambda^{4s_*-2}a_j\|_{L^p}
\\&\quad\quad+\|\nabla\Lambda^{2s_*-2}\div(au)_j\|_{L^p}
+\|(u\cdot\nabla u)_j\|_{L^p}\big)\|z_j\|_{L^p}^{p-1}.
	\end{aligned}
\end{equation*}
Let $z|_{t=0}=z_0 \triangleq
u_0+\nabla\Lambda^{2s_*-2}a_0$. Note that Lemma  \ref{difference-delta} gives rise to
\begin{equation*}
\begin{aligned}
&\|z_j\|_{L^\infty_t(L^p)} +\|z_j\|_{L^1_t(L^p)}
\\&\quad\lesssim\|(z_0)_j\|_{L^p}
+\|\nabla\Lambda^{2s_*-2}\div z_j\|_{L^1_t(L^p)}
+\|\nabla\Lambda^{4s_*-2}a_j\|_{L^1_t(L^p)}
\\&\quad\quad+\|\nabla\Lambda^{2s_*-2}\div(au)_j\|_{L^1_t(L^p)}
+\|(u\cdot\nabla u)_j\|_{L^1_t(L^p)}.
\end{aligned}
\end{equation*}
Multiplying both sides of the above inequality by $2^{j\frac dp}$, and then summing up over $j\leq J_1$, we get
\begin{equation}\label{lowz1}
\begin{aligned}
&\|z\|_{\tilde L_t^\infty
(\dot B^\frac dp_{p,1})}^{\ell}
+\|z\|_{\tilde L_t^1
(\dot B^{\frac dp}_{p,1})}^{\ell}
\\&\quad\leq C_2\big(\|z_0\|_{\dot B^\frac dp_{p,1}}^{\ell}
+\| z\|_{\tilde L_t^1
(\dot B^{\frac dp+2s_*}_{p,1})}^{\ell} 
+\|a\|_{\tilde L_t^1
(\dot B^{\frac dp+4s_*-1}_{p,1})}^{\ell}
\\&\quad\quad+\|au\|_{\tilde L_t^1
(\dot B^{\frac dp+2s_*}_{p,1})}^{\ell}
+\|u\cdot\nabla u\|_{\tilde L_t^1
(\dot B^{\frac dp}_{p,1})}^{\ell}\big).
\end{aligned}
\end{equation}
Here owing to \eqref{HLEst} and $s_*>0$, the higher order linear terms on the right hand side of \eqref{lowz1} can be analyzed by 
\begin{equation*}
\| z\|_{\tilde L_t^1
(\dot B^{\frac dp+2s_*}_{p,1})}^{\ell}
\leq2^{J_12s_*}\| z\|_{\tilde L_t^1
(\dot B^{\frac dp}_{p,1})}^{\ell},
\quad\quad
\|a\|_{\tilde L_t^1
(\dot B^{\frac dp+4s_*-1}_{p,1})}^{\ell}
\leq2^{J_12s_*}\|a\|_{\tilde L_t^1
(\dot B^{\frac dp-1+2s_*}_{p,1})}^{\ell}.
\end{equation*}
Now, multiplying \eqref{lowz1} by $2C_1$
and adding it to \eqref{lowa}, we arrive at
\begin{equation*}
\begin{aligned}
&\|a\|_{\tilde L_t^\infty
(\dot B^{\frac dp-1}_{p,1})}^{\ell}
+(1-2C_1C_22^{J_12s_*})\|a\|_{\tilde L_t^1
(\dot B^{\frac dp-1+2s_*}_{p,1})}^{\ell}\\
&\quad\quad+2C_1\|z\|_{\tilde L_t^\infty
(\dot B^\frac dp_{p,1})}^{\ell}
+C_1(1-2C_22^{J_12s_*})\|z\|_{\tilde L_t^1
(\dot B^{\frac dp}_{p,1})}^{\ell}
\\&\quad\leq C_{1}\|a_{0}\|_{
\dot B^{\frac dp-1}_{p,1}}^{\ell}
+2C_1C_2\|z_0\|_{\dot B^\frac dp_{p,1}}^{\ell}
+C_{1}\|au\|_{\tilde L_t^1
(\dot B^{\frac dp}_{p,1})}^{\ell}
\\&\quad\quad
+2C_1C_2\|au\|_{\tilde L_t^1
(\dot B^{\frac dp+2s_*}_{p,1})}^{\ell}
+2C_1C_2\|u\cdot\nabla u\|_{\tilde L_t^1
(\dot B^{\frac dp}_{p,1})}^{\ell}.
\end{aligned}
\end{equation*}
One takes the suitably small integer $J_1$ such that 
\begin{equation}\label{J_1}
2C_1C_22^{J_12s_*}<1
\quad\text{and}\quad
2C_22^{J_12s_*}<1.
\end{equation}
It thus follows that
\begin{equation}\label{lowaz}
\begin{aligned}
&\|a\|_{\tilde L_t^\infty
(\dot B^{\frac dp-1}_{p,1})}^{\ell}
+\|a\|_{\tilde L_t^1
(\dot B^{\frac dp-1+2s_*}_{p,1})}^{\ell}
+\|z\|_{\tilde L_t^\infty
(\dot B^\frac dp_{p,1})}^{\ell}
+\|z\|_{\tilde L_t^1
(\dot B^{\frac dp}_{p,1})}^{\ell}
\\&\quad\lesssim\|a_{0}\|_{
\dot B^{\frac dp-1}_{p,1}}^{\ell}
+\|z_0\|_{\dot B^\frac dp_{p,1}}^{\ell}
+\|au\|_{\tilde L_t^1
(\dot B^{\frac dp}_{p,1})}^{\ell}
+\|u\cdot\nabla u\|_{\tilde L_t^1
(\dot B^{\frac dp}_{p,1})}^{\ell}.
\end{aligned}
\end{equation}
Thanks to the definition of $z_0$, \eqref{HLEst} and $s_*>0,$ the following bounds hold:
\begin{equation*}
\begin{aligned}
\|a_{0}\|_{
\dot B^{\frac dp-1}_{p,1}}^{\ell}
+\|z_0\|_{\dot B^\frac dp_{p,1}}^{\ell}
&\lesssim\|a_{0}\|_{
\dot B^{\frac dp-1}_{p,1}}^{\ell}
+\|u_0\|_{\dot B^\frac dp_{p,1}}^{\ell}
+\|\Lambda^{2s_*-1}a_0\|_{\dot B^\frac dp_{p,1}}^{\ell} 
\\&\lesssim\|a_{0}\|_{
\dot B^{\frac dp-1}_{p,1}}^{\ell}
+\|u_0\|_{\dot B^\frac dp_{p,1}}^{\ell}
\lesssim \cE_{p,0}.
\end{aligned}
\end{equation*}
We now deal with every nonlinear term in \eqref{lowaz}.
By product laws in Lemma \ref{ClassicalProductLawEst1}, the nonlinear terms can be estimated by
\begin{equation*}
\|au\|_{\tilde L_t^1
(\dot B^{\frac dp}_{p,1})}^{\ell}
\lesssim\|u\|_{\tilde L_t^1
(\dot B^{\frac dp}_{p,1})}
\|a\|_{\tilde L_t^\infty
(\dot B^{\frac dp}_{p,1})},
\end{equation*}
and
\begin{equation*}\label{u-nabla-u}
\begin{aligned}
\|u\cdot\nabla u\|_{\tilde L_t^1
(\dot B^{\frac dp}_{p,1})}^{\ell}
\lesssim\|u\|_{\tilde L_t^\infty
(\dot B^{\frac dp}_{p,1})}
\| u\|_{\tilde L_t^1
(\dot B^{\frac dp+1}_{p,1})}.
\end{aligned}
\end{equation*}
Then, by applying the low and high frequency decomposition,  Bernstein's inequality and $s_*<1$, we obtain 
\begin{equation}\label{au-lip}
\begin{aligned}
\|u\|_{\tilde L_t^1
(\dot B^{\frac dp}_{p,1})}
+\|(a,u)\|_{\tilde L_t^1
(\dot B^{\frac dp+1}_{p,1})}
&\lesssim\|a\|_{\tilde L_t^1
(\dot B^{\frac dp+2s_*-1}_{p,1})}^\ell
+\|a\|_{\tilde L_t^1
(\dot B^{\frac d2+1}_{2,1})}^h
\\&\quad+\|u\|_{\tilde L_t^1
(\dot B^{\frac dp}_{p,1})}^\ell
+\|u\|_{\tilde L_t^1
(\dot B^{\frac d2+2-s_*}_{2,1})}^h,
\end{aligned}
\end{equation}
and
\begin{equation}\label{au-dp}
\begin{aligned}
\|(a,u)\|_{\tilde L_t^\infty
(\dot B^{\frac dp}_{p,1})}
&\lesssim\|a\|_{\tilde L_t^\infty
(\dot B^{\frac dp-1}_{p,1})}^\ell
+\|a\|_{\tilde L_t^\infty
(\dot B^{\frac d2+1}_{2,1})}^h
\\&\quad+\|u\|_{\tilde L_t^\infty
(\dot B^{\frac dp}_{p,1})}^\ell
+\|u\|_{\tilde L_t^\infty
(\dot B^{\frac d2+2-s_*}_{2,1})}^h.
\end{aligned}
\end{equation}
Substituting the above estimates into \eqref{lowaz} yields
\begin{equation}
\begin{aligned}\label{az-low-1}
\|a\|_{\tilde L_t^\infty
(\dot B^{\frac dp-1}_{p,1})}^{\ell}
+\|a\|_{\tilde L_t^1
(\dot B^{\frac dp-1+2s_*}_{p,1})}^{\ell}
+\|z\|_{\tilde L_t^\infty
(\dot B^\frac dp_{p,1})}^{\ell}
+\|z\|_{\tilde L_t^1
(\dot B^{\frac dp}_{p,1})}^{\ell}
\lesssim\cE_{p,0}+\cE_p(t)\cD_p(t).
\end{aligned}
\end{equation}
Then, together with the definition of $z$ and \eqref{az-low-1}, the estimates of $u$ can be recovered as follows
\begin{equation}\label{u-low-1}
\begin{aligned}
&\|u\|_{\tilde L_t^\infty
(\dot B^\frac dp_{p,1})}^{\ell}
+\|u\|_{\tilde L_t^1
(\dot B^{\frac dp}_{p,1})}^{\ell}
\\&\quad\lesssim\|z\|_{\tilde L_t^\infty
(\dot B^\frac dp_{p,1})}^{\ell}
+\|\Lambda^{2s_*-1}a\|_{\tilde L_t^\infty
(\dot B^\frac dp_{p,1})}^{\ell}
+\|z\|_{\tilde L_t^1
(\dot B^{\frac dp}_{p,1})}^{\ell}
+\|\Lambda^{2s_*-1}a\|_{\tilde L_t^1
(\dot B^{\frac dp}_{p,1})}^{\ell}
\\&\quad\lesssim\|a\|_{\tilde L_t^\infty
(\dot B^{\frac dp-1}_{p,1})}^{\ell}
+\|a\|_{\tilde L_t^1
(\dot B^{\frac dp-1+2s_*}_{p,1})}^{\ell}
+\|z\|_{\tilde L_t^\infty
(\dot B^\frac dp_{p,1})}^{\ell}
+\|z\|_{\tilde L_t^1
(\dot B^{\frac dp}_{p,1})}^{\ell}
\\&\quad\lesssim\cE_{p,0}+\cE_p(t)\cD_p(t).
\end{aligned}
\end{equation}

Finally, we are going to control the term $\|\partial_ta\|_{\tilde L_t^1(\dot B^{\frac dp}_{p,1})}^\ell.$ It follows from $\eqref{E-Rli}_1$, \eqref{au-lip} and \eqref{au-dp} that
\begin{equation}\label{partial-a-1}
\begin{aligned}
\|\partial_ta\|_{\tilde L_t^1(\dot B^{\frac dp}_{p,1})}^\ell
&\lesssim \|u\|_{\tilde L_t^1(\dot B^{\frac dp+1}_{p,1})}^\ell
+\|u\|_{\tilde L_t^\infty(\dot B^{\frac dp}_{p,1})}\|a\|_{\tilde L^1_t(\dot{B}^{\frac{d}{p}+1}_{p,1})}
\\&\quad\quad+\|a\|_{\tilde L_t^\infty(\dot B^{\frac dp}_{p,1})}
\| u\|_{\tilde L_t^1(\dot B^{\frac dp+1}_{p,1})}
\\&\lesssim\cE_{p,0}+\cE_p(t)\cD_p(t).
\end{aligned}
\end{equation}
The combination of \eqref{az-low-1}-\eqref{partial-a-1} gives rise to \eqref{low-global}.
\end{proof}
\subsubsection{High-frequency analysis.}
Second, we have the high-frequency estimates in the $L^{2}$ framework.
\begin{Lemma}\label{lemma-high-global}
Under the assumptions of Proposition \ref{priori-global}, it holds that
\begin{equation}\label{high-global}
\begin{aligned}
&\|a\|_{\tilde L_t^\infty
(\dot B^{\frac d2+1}_{2,1})}^{h}
+\|u\|_{\tilde L_t^\infty
(\dot B^{\frac d2+2-s_*}_{2,1})}^{h}
+\|a\|_{\tilde L_t^1
(\dot B^{\frac d2+1}_{2,1})}^{h}
+\|u\|_{\tilde L_t^1
(\dot B^{\frac d2+2-s_*}_{2,1})}^{h}
+\|\partial_ta\|_{\tilde L_t^1(\dot B^{\frac dp}_{p,1})}^h
\\&\quad\lesssim \cE_{p,0}+\cE_p(t)\cD_p(t).
\end{aligned}
\end{equation}
\end{Lemma}
\begin{proof}
In the high-frequency regime, we will employ a {\emph{hypercoercivity}} argument and take advantage of commutator estimates to avoid the loss of derivatives.
Applying $\dot\Delta_j$ to \eqref{E-Rli}, we have
\begin{equation}\label{E-Rli1}
\left\{\begin{array}{l}
\partial_ta_j+\div u_j+\dot  S_{j-1}a\div u_j 
=-\dot  S_{j-1}u\cdot \nabla a_j+R_j^1+R_j^2,\\
\partial_t u_j
+ u_j
+\nabla\Lambda^{2s_*-2}a_j
=-\dot  S_{j-1}u\cdot\nabla u_j+R_j^3
\end{array}\right.
\end{equation}
with
\begin{equation*}
\left\{\begin{array}{l}
R_{j}^1=\dot S_{j-1}u\cdot\nabla a_{j}-(u\cdot\nabla a)_{j},\\
R_{j}^2=\dot S_{j-1}a\div u_{j}-(a\div u)_{j},\\
R_{j}^3=\dot S_{j-1}u\cdot\nabla u_{j}-(u\cdot\nabla u)_{j}.
\end{array}\right.
\end{equation*}
By applying 
$\Lambda^{s_*}$
to $\eqref{E-Rli1}_{1}$ and $\Lambda$ to $\eqref{E-Rli1}_2$, one gets
\begin{equation}\label{E-Rhigh}
\left\{\begin{array}{l}
\partial_t\Lambda^{s_*}a_j+\Lambda^{s_*}\div u_j
+\Lambda^{s_*}(\dot  S_{j-1}a\div u_j) 
\\\quad\quad=-R_j^4-\dot  S_{j-1}u\cdot \nabla\Lambda^{s_*}a_j
+\Lambda^{s_*}R_j^1
+\Lambda^{s_*}R_j^2,\\
\partial_t \Lambda u_j+ \Lambda u_j
+\nabla\Lambda^{2s_*-1} a_j
=-R_j^5-\dot  S_{j-1}u\cdot \nabla\Lambda u_j
+\Lambda R_j^3
\end{array}\right.
\end{equation}
with
\begin{equation*}
\left\{\begin{array}{l}
R_j^4=[\Lambda^{s_*}, \dot  S_{j-1}u\cdot \nabla]a_j,\\
R_j^5=[\Lambda, \dot  S_{j-1}u\cdot \nabla]u_j.\\
\end{array}\right.
\end{equation*}
Multiplying both sides of $\eqref{E-Rhigh}_1-\eqref{E-Rhigh}_2$ by $\Lambda^{s_*}a_j$
and $\Lambda u_j$ respectively,  adding these together,
and then integrating the resulting equation over $\R^d$, we derive
\begin{equation}\label{dissipation-1u}
\begin{aligned}
&\frac12\frac d{dt}(\|\Lambda^{s_*}a_j\|_{L^2}^2
+\|\Lambda u_j\|_{L^2}^2)
+\|\Lambda u_j\|_{L^2}^2
+\int_{\R^{d}}\Lambda^{s_*}(\dot  S_{j-1}a\div u_j)\Lambda^{s_*}a_j\,dx
\\&\quad=-\int_{\R^{d}}R_j^4\Lambda^{s_*}a_j\,dx
-\int_{\R^{d}}\dot  S_{j-1}u\cdot \nabla\Lambda^{s_*}a_j\Lambda^{s_*}a_j\,dx
\\&\quad\quad+\int_{\R^{d}}\Lambda^{s_*}R_j^1\Lambda^{s_*}a_j\,dx
+\int_{\R^{d}}\Lambda^{s_*}R_j^2\Lambda^{s_*}a_j\,dx
-\int_{\R^{d}}R_j^5\cdot \Lambda u_j\,dx
\\&\quad\quad-\int_{\R^{d}} \dot  S_{j-1}u\cdot\nabla \Lambda u_j\cdot \Lambda u_j\,dx
+\int_{\R^{d}}\Lambda R_j^3\cdot \Lambda u_j\,dx.
\end{aligned}
\end{equation}
The difficulty is to deal with the term $\int_{\R^{d}}\Lambda^{s_*}(\dot  S_{j-1}a\div u_j)\Lambda^{s_*} a_j\,dx$. Noticing the definition of fractional power commutator and using the fact that the Riesz operator is symmetric, we have
\begin{equation*}
\begin{aligned}
&\int_{\R^{d}}\Lambda^{s_*}(\dot  S_{j-1}a\div u_j) \Lambda^{s_*}a_j\,dx\\
&\quad=\int_{\R^{d}}\dot  S_{j-1}a\div u_j\Lambda^{2s_*}a_j\,dx
\\
&\quad =-\int_{\R^{d}}\nabla\dot S_{j-1}a\cdot(\Lambda^{2s_*}a_ju_j)\,dx
-\int_{\R^{d}}\dot  S_{j-1}a u_j\cdot\nabla\Lambda^{2s_*}a_j\,dx.
\end{aligned}
\end{equation*}
To cancel the term $\int_{\R^{d}}\dot  S_{j-1}a u_j\cdot\nabla\Lambda^{2s_*}a_j\,dx$, we deduce from $\eqref{E-Rhigh}_2$ that
\begin{equation*}
\begin{aligned}
&\int_{\R^{d}}\dot  S_{j-1}a\partial_t\Lambda u_j\cdot \Lambda u_j \,dx
+\int_{\R^{d}}\dot  S_{j-1}a\Lambda u_j\cdot\Lambda u_j\,dx
  +\int_{\R^{d}} \dot S_{j-1}a\nabla\Lambda^{2s_*} a_j\cdot u_j\,dx
\\&\quad=-\int_{\R^{d}} R_j^6 u_j\,dx
-\int_{\R^{d}}\dot S_{j-1}a R_j^5\Lambda u_j\,dx
-\int_{\R^{d}}\dot  S_{j-1}a\dot  S_{j-1}u\cdot\nabla \Lambda u_j\cdot \Lambda u_j\,dx
\\&\quad\quad+\int_{\R^{d}}\dot  S_{j-1}a\Lambda R_j^3\cdot \Lambda u_j\,dx
,
\end{aligned}
\end{equation*}
where we have used 
$$R_j^6=[\Lambda, \dot S_{j-1}a\nabla]\Lambda^{2s_*-1} a_j,$$
and
\begin{equation*}
\begin{aligned}
&\int_{\R^{d}}\dot  S_{j-1}a\nabla\Lambda^{2s_*-1} a_j\cdot \Lambda u_j\,dx&=\int_{\R^{d}} R_j^6 u_j\,dx
+\int_{\R^{d}} \dot S_{j-1}a\nabla\Lambda^{2s_*} a_j\cdot u_j\,dx.
\end{aligned}
\end{equation*}
It thus follows that
\begin{equation}\label{goal for the third of-diss-a}
\begin{aligned}
&\frac12\frac d{dt}\int_{\R^{d}}\dot  S_{j-1}a(\Lambda u_j)^2\,dx
+\int_{\R^{d}}\dot  S_{j-1}a(\Lambda u_j)^2\,dx
+\int_{\R^{d}} \dot S_{j-1}a\nabla\Lambda^{2s_*} a_j\cdot u_j\,dx
\\&\quad=-\int_{\R^{d}} R_j^6 u_j\,dx
-\int_{\R^{d}}\dot  S_{j-1}a R_j^5 \Lambda u_j\,dx
-\int_{\R^{d}}\dot  S_{j-1}a\dot  S_{j-1}u\cdot\nabla\Lambda u_j\cdot \Lambda u_j\,dx
\\&\quad\quad+\int_{\R^{d}}\dot  S_{j-1}a\Lambda R_j^3\cdot \Lambda u_j\,dx
+\frac{1}{2}\int_{\R^{d}}(\Lambda u_j)^2  \partial_t\dot S_{j-1}a \,dx.
\end{aligned}
\end{equation}
Adding \eqref{dissipation-1u} and \eqref{goal for the third of-diss-a} together, we get
\begin{equation*}
\begin{aligned}
&\frac12\frac d{dt}\left(\|\Lambda^{s_*}a_j\|_{L^2}^2
+\|\Lambda u_j\|_{L^2}^2
+\int_{\R^{d}}\dot  S_{j-1}a(\Lambda u_j)^2\,dx\right)
+\|\Lambda u_j\|_{L^2}^2
\\&\quad=\int_{\R^{d}}\nabla\dot  S_{j-1}a \cdot(\Lambda^{2s_*}a_ju_j)\,dx
-\int_{\R^{d}}R_j^4\Lambda^{s_*}a_j\,dx
-\int_{\R^{d}}\dot  S_{j-1}u\cdot \nabla\Lambda^{s_*}a_j\Lambda^{s_*}a_j\,dx
\\&\quad\quad+\int_{\R^{d}}\Lambda^{s_*}R_j^1\Lambda^{s_*}a_j\,dx
+\int_{\R^{d}}\Lambda^{s_*}R_j^2\Lambda^{s_*}a_j\,dx
-\int_{\R^{d}}R_j^5\Lambda u_j\,dx
\\&\quad\quad-\int_{\R^{d}}\dot  S_{j-1}u\cdot\nabla \Lambda u_j\cdot\Lambda u_j\,dx
+\int_{\R^{d}}\Lambda R_j^3\Lambda u_j\,dx
-\int_{\R^{d}}\dot  S_{j-1}a(\Lambda u_j)^2\,dx
\\&\quad\quad-\int_{\R^{d}} R_j^6 u_j\,dx
-\int_{\R^{d}}\dot  S_{j-1}a R_j^5 \Lambda u_j\,dx
-\int_{\R^{d}}\dot  S_{j-1}a\dot  S_{j-1}u\cdot\nabla\Lambda u_j\cdot \Lambda u_j\,dx
\\&\quad\quad+\int_{\R^{d}}\dot  S_{j-1}a\Lambda R_j^3\cdot \Lambda u_j\,dx
+\frac{1}{2}\int_{\R^{d}}(\Lambda u_j)^2  \partial_t\dot S_{j-1}a \,dx.
\end{aligned}
\end{equation*}
Thanks to Bernstein's inequality and Lemma \ref{frac-commutator}, we obtain
\begin{align*}
&\int_{\R^{d}}\nabla\dot  S_{j-1}a \cdot(\Lambda^{2s_*}a_ju_j)\,dx
-\int_{\R^{d}}R_j^4\Lambda^{s_*}a_j\,dx
\\&\quad\quad-\int_{\R^{d}}\dot  S_{j-1}u\cdot \nabla\Lambda^{s_*}a_j\Lambda^{s_*}a_j\,dx
-\int_{\R^{d}}\dot  S_{j-1}u\cdot\nabla \Lambda u_j\cdot\Lambda u_j\,dx
\\&\lesssim2^{j(s_*-1)}\|\nabla a\|_{L^\infty}\|\Lambda^{s_*}a_j\|_{L^2}\|\Lambda u_j\|_{L^2}
  +\|\nabla u\|_{L^\infty}\|(\Lambda^{s_*}a_j,\Lambda u_j)\|_{L^2}^2
 \\&\lesssim\big(2^{j(s_*-1)}\|\nabla a\|_{L^\infty}
  +\|\nabla u\|_{L^\infty}\big)
  \|(\Lambda^{s_*}a_j,\Lambda u_j)\|_{L^2}^2.
\end{align*}
Similarly, it is clear that
\begin{equation}\nonumber
\begin{aligned}
&\int_{\R^{d}}R_j^5\cdot\Lambda u_j\,dx
+\int_{\R^{d}} R_j^6\cdot u_j\,dx\\
&\quad\lesssim\|\nabla u\|_{L^\infty}\|\Lambda u_j\|_{L^2}^2
+2^{j(s_*-1)}
\|\nabla a\|_{L^\infty}
\|\Lambda^{s_*}a_j\|_{L^2}
\|\Lambda u_j\|_{L^2}.
\end{aligned}
\end{equation}
Bounding other nonlinear terms is similar and therefore the details are omitted here.
%
%
For $j\geq J_1-1,$ by \eqref{a-priori-supose}, we obtain
\begin{equation}\label{diss-u1}
\begin{aligned}
&\frac12\frac d{dt}\big(\|\Lambda^{s_*}a_j\|_{L^2}^2
+\|\Lambda u_j\|_{L^2}^2
+\int_{\R^{d}}\dot  S_{j-1}a(\Lambda u_j)^2\,dx\big)
+\|\Lambda u_j\|_{L^2}^2
\\&\quad\leq C_3
\|(\nabla a,\nabla u,\partial_ta)\|_{L^\infty}
\|(\Lambda^{s_*}a_j,\Lambda u_j)\|_{L^2}^2
\\&\quad\quad+C_3\|(\Lambda^{s_*}R_{j}^{1},\Lambda^{s_*}R_{j}^{2},\Lambda R_j^3)\|_{L^2}
\|(\Lambda^{s_*}a_j,\Lambda u_j)\|_{L^2}
\\&\quad\quad+C_3\|a\|_{L^\infty}\|\Lambda u_j\|_{L^2}^2.
\end{aligned}
\end{equation}

Let us next look at the dissipation of $a$ caused by Riesz interactions. 
Applying the operator $\div$ to $\eqref{E-Rli1}_2$, multiplying of it by $a_j$,  and  multiplying $\eqref{E-Rli1}_1$ by $\div u_j$, adding them together, then integrating the resulting equation over $\R^d$,
we deduce
\begin{equation}\label{diss-a-non}
	\begin{aligned}
		&-\frac d{dt}\int_{\R^{d}}a_j\div u_j\,dx
		+\|\Lambda^{s_*} a_j\|_{L^{2}}^2
		\\&\quad=\int_{\R^{d}}a_j\div u_j\,dx
		+
		\|\div u_j\|_{L^2}^2
		+\int_{\R^{d}}a_j\div(\dot S_{j-1}u\cdot\nabla u_j)\,dx  
		\\&\quad\quad-\int_{\R^{d}}a_j\div R_j^3\,dx
		+\int_{\R^{d}}\dot S_{j-1}a(\div u_j)^2\,dx
		\\&\quad\quad
		+\int_{\R^{d}}\dot S_{j-1}u\cdot\nabla a_j \div u_j\,dx
		-\int_{\R^{d}}R_j^1 \div u_j\,dx
		-\int_{\R^{d}} R_j^2 \div u_j\,dx.
	\end{aligned}
\end{equation}
The terms on the right-hand side of \eqref{diss-a-non} are analyzed as follows.
For $j\geq J_1-1,$ by virtue of Bernstein's inequality and \eqref{HLEst}, we gain
\begin{align*}
\int_{\R^{d}}a_j\div u_j\,dx
&
\lesssim2^{-js_*}\|\Lambda^{s_*}a_j\|_{L^2}\|\Lambda u_j\|_{L^2}
\leq \var_1\|\Lambda^{s_*}a_j\|_{L^2}^2
+C\var_1^{-1}2^{-2J_1s_*}\|\Lambda u_j\|_{L^2}^2
\end{align*}
with the constant $\var_1$ being suitable small.
Using integration by parts and Bernstein's inequality, we obtain
\begin{align*}
&\int_{\R^{d}}a_j\div(\dot S_{j-1}u\cdot\nabla u_j)\,dx
+\int_{\R^{d}}\dot S_{j-1}u\cdot\nabla a_j \div u_j\,dx
\\&\quad=\int_{\R^{d}}\dot S_{j-1}u\cdot\nabla(a_j\div u_j) \,dx
+\int_{\R^{d}}a_j\nabla\dot S_{j-1}u:\nabla u_j\,dx
\\&\quad\lesssim\|\nabla u\|_{L^\infty}2^{-J_1s_*}\| \Lambda u_j\|_{L^2}
\| \Lambda^{s_*} a\|_{L^2}.
\end{align*}
By \eqref{a-priori-supose}, one can conclude for $j\geq J_1-1$ that
\begin{equation}\label{diss-a1}
\begin{aligned}
&-\frac d{dt}\int_{\R^{d}}a_j\div u_j\,dx
+\|\Lambda^{s_*} a_j\|_{L^{2}}^2
\\&\quad\leq C_4\big(\|\Lambda u_j\|_{L^2}^2
+\|\nabla u\|_{L^\infty}
\|(\Lambda^{s_*}a_j,\Lambda u_j)\|_{L^2}^2
\\&\quad\quad+\|(R_j^1,R_j^2,\Lambda R_j^3)\|_{L^2}
\|(\Lambda^{s_*}a_j,\Lambda u_j)\|_{L^2}\big).
\end{aligned}
\end{equation}

Now we define the Lyapunov functional 
\begin{align*}
\mathcal{L}_j^2(t)&\triangleq
\|\Lambda^{s_*}a_j\|_{L^2}^2
+\|\Lambda u_j\|_{L^2}^2
+\int_{\R^{d}}\dot  S_{j-1}a(\Lambda u_j)^2\,dx
-2\tilde c\int_{\R^{d}}a_j\div u_j\,dx.
\end{align*}
Based on \eqref{diss-u1} and \eqref{diss-a1}, for $j\geq J_1-1$, it holds that
\begin{equation}\label{diss-au}
\begin{aligned}
&\frac12\frac d{dt}\mathcal{L}_j^2(t)
+(1-C_3\|a\|_{L^\infty}-\tilde cC_4)\|\Lambda u_j\|_{L^2}^2
+\tilde c\|\Lambda^{s_*} a_j\|_{L^{2}}^2
\\&\quad\lesssim\|(\nabla a,\nabla u,\partial_ta)\|_{L^\infty}
\|(\Lambda^{s_*}a_j,\Lambda u_j)\|_{L^2}^2
\\&\quad\quad+\|(\Lambda^{s_*}R_{j}^{1},\Lambda^{s_*}R_{j}^{2},\Lambda R_j^3)\|_{L^2}
\|(\Lambda^{s_*}a_j,\Lambda u_j)\|_{L^2},
\end{aligned}
\end{equation}
where $\tilde{c}>0$ is suitably small and will be determined later.
 We claim that
\begin{align*}
\mathcal{L}_j^2(t)\sim
\|(\Lambda^{s_*}a_j,\Lambda u_j)\|_{L^2}^2,\quad \text{and}\quad 1-C_3\|a\|_{L^\infty}-\tilde cC_4\geq \frac{1}{2}.
\end{align*}
In fact, by H$\rm\ddot{o}$lder’s inequality and \eqref{a-priori-supose}, we can easily obtain		
$$\left|\int_{\R^{d}}\dot  S_{j-1}a
(\Lambda u_j)^2\,dx\right|
\leq C_5\|a\|_{L^\infty} \|\Lambda u_j\|_{L^2}^2\leq C_5\var_0 \|\Lambda u_j\|_{L^2}^2.$$
 It follows from Bernstein's inequality that
\begin{equation*}
\begin{aligned}
\left|2\tilde c\int_{\R^{d}}a_j\div u_j\,dx\right|
\leq \tilde{c}C_6\|(\Lambda^{s_*}a_j,\Lambda u_j)\|_{L^2}^2,
\end{aligned}
\end{equation*}
which implies that
$$(1-\tilde{c}C_6-\var_0C_5)\|(\Lambda^{s_*}a_j,\Lambda u_j)\|_{L^2}^2
\leq\mathcal{L}_j^2(t)
\leq(1+\tilde{c}C_6+\var_0C_5)\|(\Lambda^{s_*}a_j,\Lambda u_j)\|_{L^2}^2
$$
by choosing a suitably small $\tilde{c}>0$ such that
$$\tilde{c}\triangleq\min\{\frac{1}{4C_4},\frac{1}{4C_6}\},$$ 
and $$\|a\|_{L^\infty}\leq\var_0\triangleq\min\{\frac{1}{4C_3},\frac{1}{4C_5}\}.$$

Then, \eqref{diss-au} be written as
\begin{equation}\label{lyap}
\begin{aligned}
\frac d{dt} \mathcal{L}_j^2(t)
+\mathcal{L}_j^2(t)
&\lesssim\|(\nabla a,\nabla u,\partial_ta)\|_{L^\infty}\mathcal{L}_j^2(t)
\\&\quad+\|(\Lambda^{s_*}R_{j}^{1},\Lambda^{s_*}R_{j}^{2},\Lambda R_j^3)\|_{L^2}\mathcal{L}_j(t).
\end{aligned}
\end{equation}
Using Lemma \ref{difference-delta}, multiplying both sides of it by $2^{j(\frac d2+1-s_*)}$ and summing over $j\geq J_1-1$, we verify that
\begin{equation}\label{high-glo}
\begin{aligned}
&\|a\|_{\tilde L_t^\infty(\dot B^{\frac d2+1}_{2,1})}^{h}
+\|u\|_{\tilde L_t^\infty(\dot B^{\frac d2+2-s_*}_{2,1})}^{h}
+\|a\|_{\tilde L_t^1(\dot B^{\frac d2+1}_{2,1})}^{h}
+\|u\|_{\tilde L_t^1(\dot B^{\frac d2+2-s_*}_{2,1})}^{h}
\\&\quad\lesssim\|a_0\|_{\dot B^{\frac d2+1}_{2,1}}^{h}
+\|u_0\|_{\dot B^{\frac d2+2-s_*}_{2,1}}^{h}
\\&\quad\quad+(\|\nabla a\|_{\tilde L_t^\infty(\dot B^{\frac dp}_{p,1})}
+\|\nabla u\|_{\tilde L_t^\infty(\dot B^{\frac dp}_{p,1})})
\cdot(\|a\|_{\tilde L_t^1(\dot B^{\frac d2+1}_{2,1})}^{h}
+\|u\|_{\tilde L_t^1(\dot B^{\frac d2+2-s_*}_{2,1})}^{h})
\\&\quad\quad+\|\partial_ta\|_{\tilde L_t^1(\dot B^{\frac dp}_{p,1})}
(\|a\|_{\tilde L_t^\infty(\dot B^{\frac d2+1}_{2,1})}^{h}
+\|u\|_{\tilde L_t^\infty
(\dot B^{\frac d2+2-s_*}_{2,1})}^{h})
\\&\quad\quad+\sum_{j\geq J_{1}-1}2^{j(\frac{d}{2}+1)}
\|(R_{j}^{1},R_{j}^{2})\|_{L^1(L^2)}
+\sum_{j\geq J_{1}-1}2^{j(\frac d2+2-s_*)}
\|R_j^3\|_{L^1(L^2)}.
\end{aligned}
\end{equation}
Similarly to \eqref{au-dp}, one has
\begin{equation}\label{au-dp1}
\begin{aligned}
\|(\nabla a,\nabla u)\|_{\tilde L_t^\infty(\dot B^{\frac dp}_{p,1})}
&\lesssim\|a\|_{\tilde L_t^\infty(\dot B^{\frac dp-1}_{p,1})}^{\ell}
+\|u\|_{\tilde L_t^\infty(\dot B^\frac dp_{p,1})}^{\ell}
\\&\quad+\|a\|_{\tilde L_t^\infty(\dot B^{\frac d2+1}_{2,1})}^h
+\|u\|_{\tilde L_t^\infty(\dot B^{\frac d2+2-s_*}_{2,1})}^h.
\end{aligned}
\end{equation}
We now turn to handle the commutator terms.
Remembering $2\leq p\leq 4$ and  $s_*<1$ and taking $\eta_1=\frac d2+2-s_*$,
$k=0$ and $\eta_2=\frac dp+2s_*-2$ in Lemma \ref{barat-commutator-estimates}, we obtain
\begin{equation}\label{r1}
\begin{aligned}
\sum_{j\geq J_{1}-1}2^{j(\frac{d}{2}+1)}
\|R_{j}^{1}\|_{L^1(L^2)}
&\lesssim\|\nabla u\|_{\tilde L_t^1(\dot B^{\frac dp}_{p,1})}
\|\nabla a\|_{\tilde L_t^\infty(\dot B^{\frac d2}_{2,1})}^{h}
\\&\quad+2^{J_1(s_*-1)}
\|\nabla a\|_{\tilde L_t^\infty(\dot B^{\frac dp-\frac {d}{p^{\ast}}}_{p,1})}^{\ell}
\|u\|_{\tilde L_t^1(\dot B^{\frac d2+2-s_*}_{p,1})}^{\ell}
\\
&\quad+\|\nabla a\|_{\tilde L_t^\infty(\dot B^{\frac dp}_{p,1})}
\|u\|_{\tilde L_t^1(\dot B^{\frac d2+1}_{2,1})}^{h}
\\&\quad+2^{J_1(\frac{d}{p^\ast}+2-2s_*)}
\|\nabla a\|_{\tilde L_t^1(\dot B^{\frac dp+2s_*-2}_{p,1})}^{\ell}
\|\nabla u\|_{\tilde L_t^\infty(\dot B^{\frac dp-\frac {d}{p^{\ast}}}_{p,1})}^{\ell}.
\end{aligned}
\end{equation}
Here $p_*:=\frac{2p}{p-2}$. By \eqref{HLEst} and $\frac{d}{p_*}\leq 1$ due to \eqref{p}, one can arrive at
\begin{align*}
2^{J_1(s_*-1)}
\|\nabla a\|_{\tilde L_t^\infty(\dot B^{\frac dp-\frac {d}{p^{\ast}}}_{p,1})}^{\ell}
\|u\|_{\tilde L_t^1(\dot B^{\frac d2+2-s_*}_{p,1})}^{\ell}
\lesssim
\| a\|_{\tilde L_t^\infty(\dot B^{\frac dp-1}_{p,1})}^{\ell}
\|u\|_{\tilde L_t^1(\dot B^{\frac dp}_{p,1})}^{\ell}
\end{align*}
and
\begin{align*}
&2^{J_1(\frac{d}{p^\ast}+2-2s_*)}
\|\nabla a\|_{\tilde L_t^1(\dot B^{\frac dp+2s_*-2}_{p,1})}^{\ell}
\|\nabla u\|_{\tilde L_t^\infty(\dot B^{\frac dp-\frac {d}{p^{\ast}}}_{p,1})}^{\ell}
\lesssim
\|a\|_{\tilde L_t^1(\dot B^{\frac dp+2s_*-1}_{p,1})}^{\ell}
\|u\|_{\tilde L_t^\infty(\dot B^{\frac dp}_{p,1})}^{\ell}.
\end{align*}
Thus, recalling \eqref{au-lip} and \eqref{au-dp1}, we get
\begin{equation}\label{R-1}
\sum_{j\geq J_{1}-1}2^{j(\frac{d}{2}+1)}
\|R_{j}^{1}\|_{L^1(L^2)}
\lesssim \cE_p(t)\cD_p(t).
\end{equation}
Similar computations yield
\begin{equation}\label{R2-3}
\begin{aligned}
\sum_{j\geq J_{1}-1}2^{j(\frac{d}{2}+1)}
\|R_{j}^{2}\|_{L^1(L^2)}
+\sum_{j\geq J_{1}-1}2^{j(\frac d2+2-s_*)}
\|R_j^3\|_{L^1(L^2)}
\lesssim \cE_p(t)\cD_p(t).
\end{aligned}
\end{equation}
Inserting \eqref{au-dp1}, \eqref{R-1} and \eqref{R2-3} into \eqref{high-glo},
we conclude that
\begin{equation*}
\begin{aligned}
&\|a\|_{\tilde L_t^\infty(\dot B^{\frac d2+1}_{2,1})}^{h}
+\|u\|_{\tilde L_t^\infty(\dot B^{\frac d2+2-s_*}_{2,1})}^{h}
+\|a\|_{\tilde L_t^1(\dot B^{\frac d2+1}_{2,1})}^{h}
+\|u\|_{\tilde L_t^1(\dot B^{\frac d2+2-s_*}_{2,1})}^{h}
\\&\quad\lesssim \cE_{p,0}+\cE_p(t)\cD_p(t).
\end{aligned}
\end{equation*}
The term $\|\partial_ta\|_{\tilde L_t^1(\dot B^{\frac dp}_{p,1})}^h$ can be bounded using the same procedures as \eqref{partial-a-1}, 
 so we get
\begin{equation*}
\begin{aligned}
\|\partial_ta\|_{\tilde L_t^1(\dot B^{\frac dp}_{p,1})}^h
\lesssim\cE_{p,0}+\cE_p(t)\cD_p(t).
\end{aligned}
\end{equation*}
Collecting the above estimates, we get \eqref{high-global}.

\vspace{3mm}

Combining \eqref{low-global} and \eqref{high-global}, we obtain
\begin{equation*}
\begin{aligned}
\cE_p(t)+\cD_p(t)
\leq C_0\big(
\cE_{p,0}+\cE_p(t)\cD_p(t)
\big).
\end{aligned}
\end{equation*}
This completes the proof of Proposition \ref{priori-global}.
\end{proof}

\subsection{Proof of global existence and uniqueness}

Before proving Theorem \ref{global}, we state the existence and uniqueness of local-in-time
solutions for the Cauchy problem \eqref{E-Rli} (see another paper \cite{csx} for details).
\begin{theorem}\label{local}
\rm(Local well-posedness\rm) 
Let $d\geq1$, $p$ satisfy \eqref{p} and $s_*\triangleq\frac{\alpha-d+2}{2}\in (0,1)$.
Assume $a_0\in\dot B^{\frac dp-1,\frac{d}{2}+1}_{p,2}$, $u_0\in
\dot B^{\frac dp,\frac{d}{2}+2-s_*}_{p,2}$ and that $1+a_0$ is away from $0$. Then, there exists a time $T>0$ such that the Cauchy problem \eqref{E-Rli} admits a unique strong solution $(a,u)$ satisfying
\begin{equation}\label{local-space}
\left\{
\begin{aligned}
& \inf_{(t,x)\in [0,T)\times\mathbb{R}^d}(1+a)>0,\\
&a \in \cC([0,T); \dot B^{\frac dp-1,\frac{d}{2}+1}_{p,2}),\quad 
u\in \cC([0,T); \dot B^{\frac{d}{p},\frac{d}{2}+2-s_*}_{p,2}).
\end{aligned}
\right.
\end{equation}
\end{theorem}

\vspace{2mm}

\noindent
\textbf{Proof of Theorem \ref{global}}. Theorem \ref{local} ensures that there
exists a maximal existence time $T_0$ such that the Cauchy problem \eqref{E-Rli} has a unique solution $(a,u)$ satisfying \eqref{local-space}.
Let
\begin{equation*}
\begin{aligned}
\cY_p(t)&\triangleq
\|a(t)\|_{\dot B^{\frac dp-1}_{p,1}}^{\ell}
+\|u(t)\|_{\dot B^\frac dp_{p,1}}^{\ell}
+\|a(t)\|_{\dot B^{\frac d2+1}_{2,1}}^h
+\|u(t)\|_{\dot B^{\frac d2+2-s_*}_{2,1}}^h
+\cD_p(t)
\end{aligned}
\end{equation*}
with $\cD_p(t)$ given by \eqref{cDpt}. 
Define
$$T_\ast\triangleq\sup\{t\in[0,T_0)\,\,|\,\,\cY_p(t)\leq2C_0\cE_{p,0}\},$$
where $C_0$ is given by Proposition \ref{priori-global}.

We first claim $T_\ast=T_0$.
Indeed, if $T_\ast<T_0,$ then due to the embedding 
$\dot{B}_{p,1}^\frac dp\hookrightarrow L^\infty$,
one has 
$$\|a\|_{L^\infty_t(L^\infty)}\leq
C(\|a\|_{L_t^\infty(\dot{B}_{p,1}^{\frac dp-1})}^\ell
+\|a\|_{L_t^\infty(\dot{B}_{2,1}^{\frac d2+1})}^h)
\leq2CC_0\cE_{p,0},
\,\,0<t<T_\ast.$$
Thus, by a  priori estimate \eqref{priori-estimate} established in Proposition \ref{priori-global},
for $0<t<T_\ast$,
we have
\begin{equation*}
\begin{aligned}
&\cE_p(t)+\cD_p(t)
\leq C_0\big(
\cE_{p,0}+\cE_p(t)2C_0\cE_{p,0}
\big).
\end{aligned}
\end{equation*}
By letting
$$\cE_{p,0}\leq \delta_0\triangleq\min\{\frac{\var_0}{2CC_0},
\frac1{6C_0^2}\},$$
we get
\begin{equation*}
\begin{aligned}
\cE_p(t)+\cD_p(t)
\leq\frac32C_0\cE_{p,0}.
\end{aligned}
\end{equation*}
In this case, the time continuity of $\cY_p(t)$ implies that
$$\cY_p(T_\ast)\leq\frac32C_0\cE_{p,0},$$
which contradicts the definition of $T_\ast$.
Thus, $T_\ast=T_0$ follows.

Finally, we prove  $T_0=+\infty,$ and therefore
$(a,u)$ is a global solution to the Cauchy problem \eqref{E-Rli} satisfying
$$\cE_p(t)+\cD_p(t)\leq2C_0\cE_{p,0},\quad \text{for all}\quad
t>0.$$
In fact, if $T_0<+\infty,$ then one can take 
$(a,u)(t_0,x)$ with $t_0$ sufficiently
close to $T_0$ as a new initial datum and extend a 
solution beyond $T_0$ due to Theorem \ref{local}. 
This contradicts the definition of $T_0$.
The proof of Theorem \ref{global} is thus finished.
%
\section{Proof of Theorem \ref{decay-a-u}}\label{decay-proof}
This section we focus on the decay estimates for the pressureless Euler–Riesz system.
\subsection{The evolution of the low-frequency Besov norm}
In this section, inspired by Guo \& Wang \cite{guo-wang} and Xin \& Xu \cite{xin-xu}, we establish the evolution of  the $\dot B_{p,\infty}^{\sigma_1}$-norm for low frequencies in the $L^p$ framework, 
which is the main ingredient to derive the decay estimates.
\begin{proposition}\label{negative besov bounded}
 Let $(a,u)$ be the global solution to the Cauchy problem \eqref{E-Rli} obtained in Theorem \ref{global}.
 If in addition to \eqref{x-p-0}, suppose further 
$a_0^\ell\in \dot B^{\sigma_1}_{p,\infty}$ and $u_0^\ell\in
\dot B^{\sigma_1+1}_{p,\infty}$ with $-\frac dp-1\leq\sigma_1<\frac dp-1$ such that $\cX_{p,0}$ given by \eqref{Xp0}
 is bounded.
	
Then, it holds that
\begin{equation}\label{negtative besov}
\begin{aligned}
&\|a^\ell\|_{\tilde L^\infty
(\R_+;\dot B^{\sigma_1}_{p,\infty})}
+\|u^\ell\|_{\tilde L^\infty
(\R_+;\dot B^{\sigma_1+1}_{p,\infty})}
+\|a^\ell\|_{\tilde L^1
(\R_+;\dot B^{\sigma_1+2s_*}_{p,\infty})}
\\&\quad+\|u^\ell\|_{\tilde L^1
(\R_+;\dot B^{\sigma_1+1}_{p,\infty})}
\lesssim \cX_{p,0}.
\end{aligned}
\end{equation}
\end{proposition}
\begin{proof}
Applying $\dot\Delta_j\dot S_{J_1}$ to $\eqref{low-a}$, 
we obtain
\begin{equation*}
\begin{array}{l}
\partial_ta_j^\ell
+\Lambda^{2s_*}a_j^\ell
=-\div z_j^\ell -\div(a u)_j^\ell.
\end{array}
\end{equation*}
Arguing similarly to \eqref{lowa}, one has
\begin{equation*}
\begin{aligned}
&\|a^\ell\|_{\tilde L^\infty
(\R_+;\dot B^{\sigma_1}_{p,\infty})}
+\|a^{\ell}\|_{\tilde L^1
(\R_+;\dot B^{\sigma_1+2s_*}_{p,\infty})}
\\&\quad\leq C_1\big(\|a_{0}^{\ell}\|_{
\dot B^{\sigma_1}_{p,\infty}}
+\|z^{\ell}\|_{\tilde L^1
(\R_+;\dot B^{\sigma_1+1}_{p,\infty})}
+\|(a u)^{\ell}\|_{\tilde L^1
(\R_+;\dot B^{\sigma_1+1}_{p,\infty})}
\big).
\end{aligned}
\end{equation*}
Recall the effective velocity $z\triangleq u+\nabla\Lambda^{2s_*-2}a.$ Applying $\dot\Delta_j\dot S_{J_1}$ to \eqref{low-z},
we obtain
\begin{align*}
\partial_t z_j^\ell+ z_j^\ell=-\nabla\Lambda^{2s_*-2}\div z_j^\ell
-\nabla\Lambda^{4s_*-2}a_j^\ell-\nabla\Lambda^{2s_*-2}\div(au)_j^\ell
-(u\cdot\nabla u)_j^\ell.
\end{align*}
Similarly, we arrive at
\begin{equation*}\begin{aligned}
&\|z^\ell\|_{\tilde L^\infty
(\R_+;\dot B^{\sigma_1+1}_{p,\infty})}
+\|z^\ell\|_{\tilde L^1
(\R_+;\dot B^{\sigma_1+1}_{p,\infty})}
\\&\quad\leq C_2\big(\|z_0^\ell\|_{
\dot B^{\sigma_1+1}_{p,\infty}}
+2^{J_12s_*}\|z^\ell\|_{\tilde L^1
(\R_+;\dot B^{\sigma_1+1}_{p,\infty})}
+2^{J_12s_*}\|a^\ell\|_{\tilde L^1
(\R_+;\dot B^{\sigma_1+2s_*}_{p,\infty})}
\\&\quad\quad+\|(au)^\ell\|_{\tilde L^1
(\R_+;\dot B^{\sigma_1+2s_*+1}_{p,\infty})}
+\|(u\cdot\nabla u)^\ell\|_{\tilde L^1
(\R_+;\dot B^{\sigma_1+1}_{p,\infty})}\big).
\end{aligned}
\end{equation*}
Here the constants $C_1$ and $C_2$ are exactly the same as those in the estimates $\eqref{lowa}$ and $\eqref{lowz1}$.
Then, as $J_1$ is given by \eqref{J_1}, it holds that
\begin{equation}\label{neg-nonlin}
\begin{aligned}
&\|a^\ell\|_{\tilde L^\infty
(\R_+;\dot B^{\sigma_1}_{p,\infty})}
+\|z^\ell\|_{\tilde L^\infty
(\R_+;\dot B^{\sigma_1+1}_{p,\infty})}\\
&\quad\quad+\|a^\ell\|_{\tilde L^1
(\R_+;\dot B^{\sigma_1+2s_*}_{p,\infty})}
+\|z^\ell\|_{\tilde L^1
(\R_+;\dot B^{\sigma_1+1}_{p,\infty})}
\\&\quad\lesssim\|a_{0}^\ell\|_{
\dot B^{\sigma_1}_{p,\infty}}
+\|z_{0}^\ell\|_{
\dot B^{\sigma_1+1}_{p,\infty}}\\
&\quad\quad+\|(a u)^\ell\|_{\tilde L^1
(\R_+;\dot B^{\sigma_1+1}_{p,\infty})}
+\|(u\cdot\nabla u)^\ell\|_{\tilde L^1
(\R_+;\dot B^{\sigma_1+1}_{p,\infty})}.
\end{aligned}
\end{equation}
Let $-\frac dp-1\leq\sigma_1<\frac dp-1$ and $p\geq2$. 
According to Lemma \ref{negative product estimate} and \eqref{HLEst}, we arrive at
\begin{align*}
\|a u\|_{\tilde L^1
(\R_+;\dot B^{\sigma_1+1}_{p,\infty})}
\lesssim\|a\|_{\tilde L^\infty
(\R_+;\dot B^{\frac dp}_{p,1})}
\|u\|_{\tilde L^1
(\R_+;\dot B^{\sigma_1+1}_{p,\infty})}.
\end{align*}
Taking advantage of the embedding properties in Lemma \ref{embedding}, the definition of $z$ and \eqref{HLEst}, we have
\begin{equation*}
\begin{aligned}
\|u\|_{\tilde L^1
(\R_+;\dot{B}^{\sigma_1+1}_{p,\infty})}
&\lesssim\|z^\ell\|_{\tilde L^1
(\R_+;\dot{B}^{\sigma_1+1}_{p,\infty})}
+\|\nabla\Lambda^{2s_*-2}a^\ell\|_{\tilde L^1
(\R_+;\dot{B}^{\sigma_1+1}_{p,\infty})}
+\|u\|_{\tilde L^1
(\R_+;\dot{B}^{\sigma_1+1}_{p,\infty})}^h
\\&\lesssim\|z^\ell\|_{\tilde L^1
(\R_+;\dot{B}^{\sigma_1+1}_{p,\infty})}
+\|a^\ell\|_{\tilde L^1
(\R_+;\dot{B}^{\sigma_1+2s_*}_{p,\infty})}
+\|u\|_{\tilde L^1
(\R_+;\dot{B}^{\frac d2+2-s_*}_{2,1})}^h.
\end{aligned}
\end{equation*}
Therefore, thanks to \eqref{global-Var0} and \eqref{au-dp}, we obtain
\begin{equation}\label{au-sigma1-1}
\begin{aligned}
\|a u\|_{\tilde L^1
(\R_+;\dot B^{\sigma_1+1}_{p,\infty})}
\lesssim\cE_{p,0}\big(\|a^\ell\|_{\tilde L^1
(\R_+;\dot{B}^{\sigma_1+2s_*}_{p,\infty})}
+\|z^\ell\|_{\tilde L^1
(\R_+;\dot{B}^{\sigma_1+1}_{p,\infty})}\big)
+\cE_{p,0}^2.
\end{aligned}
\end{equation}
 Similarly, it is clear that
\begin{equation}\label{u-nabla-u-neg}
\begin{aligned}
\|u\cdot\nabla u\|_{\tilde L^1
(\R_+;\dot B^{\sigma_1+1}_{p,\infty})}
&\lesssim\|u\|_{\tilde L^1
(\R_+;\dot B^{\sigma_1+1}_{p,\infty})}
\|u\|_{\tilde L^\infty
(\R_+;\dot B^{\frac dp+1}_{p,1})}
\\&\lesssim\cE_{p,0}\big(\|a^\ell\|_{\tilde L^1
(\R_+;\dot{B}^{\sigma_1+2s_*}_{p,\infty})}
+\|z^\ell\|_{\tilde L^1
(\R_+;\dot{B}^{\sigma_1+1}_{p,\infty})}\big)
+\cE_{p,0}^2.
\end{aligned}
\end{equation}
Inserting \eqref{au-sigma1-1} and \eqref{u-nabla-u-neg} into \eqref{neg-nonlin}, 
Due to the smallness of $\cE_{p,0}$, it holds that
\begin{equation*}
\begin{aligned}
&\|a^\ell\|_{\tilde L^\infty
(\R_+;\dot B^{\sigma_1}_{p,\infty})}
+\|z^\ell\|_{\tilde L^\infty
(\R_+;\dot B^{\sigma_1+1}_{p,\infty})}
+\|a^\ell\|_{\tilde L^1
(\R_+;\dot B^{\sigma_1+2s_*}_{p,\infty})}
+\|z^\ell\|_{\tilde L^1
(\R_+;\dot B^{\sigma_1+1}_{p,\infty})}
\\&\quad\lesssim\|a_{0}^\ell\|_{
\dot B^{\sigma_1}_{p,\infty}}
+\|z^\ell_0\|_{\tilde L^\infty
(\R_+;\dot B^{\sigma_1+1}_{p,\infty})}
+\cE_{p,0}.
\end{aligned}
\end{equation*}
 By the low-high-frequency decomposition, Berstein's inequality, \eqref{HLEst} and \eqref{low-1-infty}, we have
\begin{align*}
&\cE_{p,0}
\lesssim\|a_{0}\|_{\dot B^{\sigma_1}_{p,\infty}}^\ell
+\|u_{0}\|_{\dot B^{\sigma_1+1}_{p,\infty}}^\ell
+\|a_0\|_{\dot B^{\frac d2+1}_{2,1}}^{h}
+\|u_0\|_{\dot B^{\frac{d}{2}+2-s_*}_{2,1}}^{h}
\\&\quad\lesssim\|a_{0}^\ell\|_{\dot B^{\sigma_1}_{p,\infty}}
+\|u_{0}^\ell\|_{\dot B^{\sigma_1+1}_{p,\infty}}
+\|a_0\|_{\dot{B}^{\frac{d}{2}+1}_{2,1}}^{h}
+\|u_0\|_{\dot B^{\frac{d}{2}+2-s_*}_{2,1}}^{h}
=\cX_{p,0}.
\end{align*}
Thus, we arrive at 
\begin{equation*}
\begin{aligned}
&\|a^\ell\|_{\tilde L^\infty
(\R_+;\dot B^{\sigma_1}_{p,\infty})}
+\|z^\ell\|_{\tilde L^\infty
(\R_+;\dot B^{\sigma_1+1}_{p,\infty})}
+\|a^\ell\|_{\tilde L^1
(\R_+;\dot B^{\sigma_1+2s_*}_{p,\infty})}
+\|z^\ell\|_{\tilde L^1
(\R_+;\dot B^{\sigma_1+1}_{p,\infty})}
\lesssim \cX_{p,0}.
\end{aligned}
\end{equation*}
Then, together with the definition of $z$ and \eqref{HLEst}, the estimates of $u$ can be recovered as in \eqref{negtative besov}.
\subsection{Time-decay estimates}
Define
$$E(t)=\|a^\ell\|_{\dot B^{\frac dp-1}_{p,1}}
+\|u^{\ell}\|_{\dot B^\frac dp_{p,1}}
+\|a\|_{\dot B^{\frac d2+1}_{2,1}}^{h}
+
\|u\|_{\dot B^{\frac d2+2-s_*}_{2,1}}^{h},
$$
and
$$D(t)=\|a^{\ell}\|_{\dot B^{\frac dp+2s_*-1}_{p,1}}
+
\|u^{\ell}\|_{\dot B^{\frac dp}_{p,1}}
+
\|a\|_{\dot B^{\frac d2+1}_{2,1}}^{h}
+
\|u\|_{\dot B^{\frac d2+2-s_*}_{2,1}}^{h}
+\|\partial_ta\|_{\dot B^{\frac dp}_{p,1}}.$$
By currying out similar computations as in Lemmas \ref{lemma-low-global} and \ref{lemma-high-global},
we are able to deduce the following estimates
\begin{equation*}
\begin{aligned}
&E(t)+\int_{t_0}^tD(\tau)d\tau
\lesssim 
E(t_0)
+ \cE_p(t)\int_{t_0}^tD(\tau)d\tau,
\quad t\geq t_0\geq 0.
\end{aligned}
\end{equation*}
For simplicity, we omit the details here.
Due to the smallness of $\cE_{p,0}$ and \eqref{global-Var0}, for $t\geq t_0,$ it holds that
\begin{equation*}
\begin{aligned}
&E(t)+\int_{t_0}^tD(\tau)d\tau
\lesssim 
E(t_0)
\quad t\geq t_0\geq 0.
\end{aligned}
\end{equation*}
In particular, we get
\begin{equation*}
\begin{aligned}
&E(t+h)+\int_{t_0}^{t+h}D(\tau)d\tau
\lesssim 
E(t_0).
\end{aligned}
\end{equation*}
Thus, for any $h>0$, one finds out that
\begin{align*}
\frac{E(t+h)-E(t)}{h}
+\frac1h\int_{t}^{t+h}D(\tau)d\tau
\leq 0, \quad \text{a.e.\,\,on\,\,}\R_+.
\end{align*}
Since  $E$ is nonincreasing on $\R_+$, $E$ is differentiable almost everywhere and satisfies
$$
\frac d{dt}E(t)+D(t)
\leq 0 \quad
\text{a.e.\,\,on\,\,}\R_+.
$$

Based on Lemma \ref{Classical Interpolation} and a interpolation argument, the dissipation $D(t)$ can control some power of the functional $E(t)$, which is the key to deriving the time-decay estimates. For $\sigma_1<\frac dp-1$, using the real
interpolation in Lemma \ref{Classical Interpolation}, we have
\begin{equation*}
\begin{aligned}
\|a^\ell\|_{\dot{B}_{p,1}^{\frac dp-1}}
\lesssim\|a^\ell\|_{\dot{B}_{p,\infty}^{\sigma_1}}^{\theta_1}
\|a^\ell\|_{\dot{B}_{p,\infty}^{\frac dp+2s_*-1}}^{1-\theta_1}
\end{aligned}
\end{equation*}
with $\theta_1=\frac{2s_*}{2s_*+\frac dp-1-\sigma_1}\in(0,1).$
Due to Proposition \ref{negative besov bounded}, it holds
\begin{equation*}\label{-sigma-bounded}
\begin{aligned}
\|a^\ell\|_{\dot{B}_{p,\infty}^{\sigma_1}}
\lesssim \cX_{p,0}.
\end{aligned}
\end{equation*}
Thus, one gets 
\begin{equation*}
\begin{aligned}
\|a^\ell\|_{\dot{B}_{p,1}^{\frac dp+2s_*-1}}
\geq C^{\frac{-1}{1-\theta_1}}(\cX_{p,0}
+\eta)^{\frac {-\theta_1}{1-\theta_1}}
\|a^\ell\|_{\dot{B}_{p,1}^{\frac dp-1}}^{\frac 1{1-\theta_1}}.
\end{aligned}
\end{equation*}
Here $\eta>0$ is any given constant. In addition, noticing that
\begin{equation*}
\begin{aligned}
\|u^\ell\|_{\dot{B}_{p,1}^\frac dp}
\lesssim\|u^\ell\|_{\dot{B}_{p,1}^{\frac dp}}^{\theta_1}
\|u^\ell\|_{\dot{B}_{p,1}^{\frac dp}}^{1-\theta_1},
\end{aligned}
\end{equation*}
we also get
\begin{equation*}
\begin{aligned}
\|u^\ell\|_{\dot{B}_{p,1}^{\frac dp}}
\geq C^{\frac{-1}{1-\theta_1}}(\cX_{p,0}
+\eta)^{\frac {-\theta_1}{1-\theta_1}}
\|u^\ell\|_{\dot{B}_{p,1}^\frac dp}^{\frac 1{1-\theta_1}}.
\end{aligned}
\end{equation*}
Similarly, it holds that
\begin{equation*}
\begin{aligned}
\|a\|_{\dot{B}_{2,1}^{\frac d2+1}}^h
+\|u\|_{\dot{B}_{2,1}^{\frac d2+2-s_*}}^h
&\geq C^{\frac{-1}{1-\theta_1}}(\cX_{p,0}
+\eta)^{\frac {-\theta_1}{1-\theta_1}}
(\|a\|_{\dot{B}_{2,1}^{\frac d2+1}}^h
+\|u\|_{\dot{B}_{2,1}^{\frac d2+2-s_*}}^h)^{\frac 1{1-\theta_1}}.
\end{aligned}
\end{equation*}
Consequently, the following Lyapunov-type inequality
holds
\begin{equation*}
\begin{aligned}
&\frac d{dt}E(t)+\big(\cX_{p,0}
+\eta\big)^{-\frac{2s_*}{\frac dp-1-\sigma_1}}
E(t)^{1+\frac{2s_*}{\frac dp-1-\sigma_1}}
\leq 0.
\end{aligned}
\end{equation*}
Let $E_0\triangleq E(t)|_{t=0}\lesssim\cX_{p,0}$. Solving this differential inequality leads to 
\begin{equation*}
\begin{aligned}
&\|a^\ell\|_{\dot{B}_{p,1}^{\frac dp-1}}
+\|u^\ell\|_{\dot{B}_{p,1}^\frac dp}
+\|a\|_{\dot{B}_{2,1}^{\frac d2+1}}^h
+\|u\|_{\dot{B}_{2,1}^{\frac d2+2-s_*}}^h
\\&\quad\leq \Big({E_0}^{-\frac{2s_*}{\frac dp-1-\sigma_1}}
+\frac{2s_*}{\frac dp-1-\sigma_1}
(\cX_{p,0}+\eta)^{-\frac{2s_*}{\frac dp-1-\sigma_1}}t\Big)^{-\frac{1}{2s_*}(\frac dp-1-\sigma_1)}
\\&\quad\lesssim(\cX_{p,0}+\eta)(1+t)^{-\frac{1}{2s_*}(\frac dp-1-\sigma_1)}.
\end{aligned}
\end{equation*}
As $\eta \to 0$, we obtain
\begin{equation*}
\begin{aligned}
&\|a^\ell\|_{\dot{B}_{p,1}^{\frac dp-1}}
+\|u^\ell\|_{\dot{B}_{p,1}^\frac dp}
+\|a\|_{\dot{B}_{2,1}^{\frac d2+1}}^h
+\|u\|_{\dot{B}_{2,1}^{\frac d2+2-s_*}}^h
\lesssim\cX_{p,0}(1+t)^{-\frac{1}{2s_*}(\frac dp-1-\sigma_1)}
\end{aligned}
\end{equation*}
for all $t\geq0$ and $-\frac dp-1\leq\sigma_1<\frac dp-1$. 
By Bernstein's inequality and \eqref{HLEst}, we arrive at
\begin{equation}\label{au-fracdp-decay}
\begin{aligned}
\|a\|_{\dot{B}_{p,1}^{\frac dp-1}}
+\|u\|_{\dot{B}_{p,1}^{\frac dp}}
&\lesssim
\|a^\ell\|_{\dot{B}_{p,1}^{\frac dp-1}}
+\|u^\ell\|_{\dot{B}_{p,1}^\frac dp}
+\|a\|_{\dot{B}_{2,1}^{\frac d2+1}}^h
+\|u\|_{\dot{B}_{2,1}^{\frac d2+2-s_*}}^h
\\&\lesssim\cX_{p,0}(1+t)^{-\frac{1}{2s_*}(\frac dp-1-\sigma_1)}.
\end{aligned}
\end{equation}
In addition, if
$\sigma_1<\frac dp-1$, then it holds from Lemma \ref{Classical Interpolation} that
\begin{equation}\label{a-sigma-Interpolation}
\begin{aligned}
&\|a^\ell\|_{\dot{B}_{p,1}^\sigma}
\lesssim\|a^\ell\|_{\dot{B}_{p,\infty}^{\sigma_1}}^{\theta_2}
\|a^\ell\|_{\dot{B}_{p,\infty}^{\frac dp-1}}^{1-\theta_2},
\quad\sigma\in(\sigma_1,\frac d{p}-1),
\end{aligned}
\end{equation}
with
$\theta_2=\frac{\frac dp-1-\sigma}{\frac dp-1-\sigma_1}\in(0,1)$.
By virtue of \eqref{negtative besov}, \eqref{au-fracdp-decay} and \eqref{a-sigma-Interpolation}, for all $t>0$, $-\frac dp-1\leq\sigma_1<\frac dp-1$, we gain
\begin{equation*}
\begin{aligned}
\|a^\ell\|_{\dot{B}_{p,1}^\sigma}
\lesssim\cX_{p,0}
(1+t)^{-\frac{1}{2s_*}(\sigma-\sigma_1)},
\quad\sigma\in(\sigma_1,\frac d{p}-1).
\end{aligned}
\end{equation*}
This, together with the high-frequency decay in \eqref{au-fracdp-decay}, 
leads to
\begin{equation*}
\begin{aligned}
\|a\|_{\dot{B}_{p,1}^\sigma}
\lesssim\cX_{p,0}
(1+t)^{-\frac{1}{2s_*}(\sigma-\sigma_1)},
\quad\sigma\in(\sigma_1,\frac d{p}-1].
\end{aligned}
\end{equation*}
Moreover,
it follows from \eqref{au-fracdp-decay} that
\begin{equation}\label{a-decay-dp1}
\begin{aligned}
\|a\|_{\dot{B}_{p,1}^\sigma}
 &\lesssim\|a^\ell\|_{\dot{B}_{p,1}^{\frac dp-1}}
 +\|a\|_{\dot{B}_{2,1}^{\frac d2+1}}^h\\
&\lesssim \cX_{p,0}(1+t)^{-\frac{1}{2s_*}(\frac{d}{p}-1-\sigma_1)},
\quad
\sigma\in(\frac d{p}-1,\frac{d}{p}+1].
\end{aligned}
\end{equation}
Therefore, we prove the decay estimate $\eqref{decayau}$.

Similarly, we also get
\begin{equation}\label{u-sigma-Interpolation}
\begin{aligned}
&\|u^\ell\|_{\dot{B}_{p,1}^{\sigma}}
\lesssim\|u^\ell\|_{\dot{B}_{p,\infty}^{\sigma_1+1}}^{\theta_3}
\|u^\ell\|_{\dot{B}_{p,\infty}^{\frac dp}}^{1-\theta_3},
\quad\sigma\in(\sigma_1+1,\frac d{p})
\end{aligned}
\end{equation}
with $\theta_3=\frac{\frac dp-\sigma}{\frac dp-1-\sigma_1}\in(0,1)$.
Using \eqref{negtative besov}, \eqref{au-fracdp-decay} and \eqref{u-sigma-Interpolation}, for all $t>0$, $-\frac dp-1\leq\sigma_1<\frac dp-1$, we obtain
\begin{equation*}
\begin{aligned}
\|u^\ell\|_{\dot{B}_{p,1}^{\sigma}}
\lesssim\cX_{p,0}
(1+t)^{-\frac{1}{2s_*}(\sigma-\sigma_1-1)},
\quad\sigma\in(\sigma_1+1,\frac d{p}).
\end{aligned}
\end{equation*}
Then one can arrive at
\begin{equation}\label{u-decay-1}
\begin{aligned}
\|u\|_{\dot{B}_{p,1}^{\sigma}}
\lesssim\cX_{p,0}
(1+t)^{-\frac{1}{2s_*}(\sigma-\sigma_1-1)},
\quad\sigma\in(\sigma_1+1,\frac d{p}],
\end{aligned}
\end{equation}
and
\begin{equation}\label{u-decay-2}
\begin{aligned}
\|u\|_{\dot{B}_{p,1}^\sigma}
 &\lesssim\|u^\ell\|_{\dot{B}_{p,1}^{\frac dp}}
 +\|u\|_{\dot{B}_{2,1}^{\frac d2+2-s_*}}^h\\
&\lesssim \cX_{p,0}(1+t)^{-\frac{1}{2s_*}(\frac{d}{p}-1-\sigma_1)},
\quad
\sigma\in(\frac dp,\frac dp+2-s_*].
\end{aligned}
\end{equation}
The decay of $u$ and $(a^h,u^h)$ can be improved below.
 \subsection{Improved time decay rates of \texorpdfstring{$u$}{u} and \texorpdfstring{$(a^h,u^h)$}{(ah, uh)}}
First, we present the proof of $\eqref{decayau1}$.
Recall $\eqref{E-Rli}_2$ that
$$\partial_t u+u+\nabla\Lambda^{2s_*-2}a
=-u\cdot\nabla u.$$
Performing a routine procedure yields
\begin{equation*}
\begin{aligned}
\|u\|_{\dot{B}_{p,1}^{\sigma}}
&\lesssim
e^{-t}\|u_0\|_{\dot{B}_{p,1}^{\sigma}}
+\int_{0}^te^{-(t-\tau)}\|a\|_{\dot{B}_{p,1}^{\sigma+2s_*-1}}d\tau
+\int_{0}^te^{-(t-\tau)}\|u\cdot\nabla u\|_{\dot{B}_{p,1}^{\sigma}}d\tau.
\end{aligned}
\end{equation*}
According to Lemma \ref{embedding}, \eqref{HLEst}, \eqref{low-1-infty} and \eqref{Xp0}, one arrives at
$$\|u_0\|_{\dot{B}_{p,1}^{\sigma}}
\lesssim\|u_0^\ell\|_{\dot{B}_{p,\infty}^{\sigma_1+1}}
+\|u_0\|_{\dot{B}_{2,1}^{\frac d2+2-s_*}}^h
\leq \cX_{p,0},
\quad\sigma\in(\sigma_1+1,\frac dp+2-s_*].$$
If $\sigma_1+1<\sigma\leq\frac dp-2s_*$, 
it follows from $\eqref{decayau}_1$ that
\begin{equation}\label{a-decay}
\|a\|_{\dot{B}_{p,1}^{\sigma+2s_*-1}}\lesssim\cX_{p,0}
(1+t)^{-\frac{1}{2s_*}(\sigma+2s_*-1-\sigma_1)}.
\end{equation}
Then, if $\sigma\in(\sigma_1+1,\frac dp]$, we deduce from \eqref{u-decay-1} and \eqref{u-decay-2} that
\begin{align}\label{u-nabla-u-decay}
\|u\cdot\nabla u\|_{\dot{B}_{p,1}^{\sigma}}
\lesssim\|u\|_{\dot{B}_{p,1}^{\frac dp+1}}
\|u\|_{\dot{B}_{p,1}^{\sigma}}
\lesssim\cX_{p,0}^{2}(1+t)^{-\frac{1}{2s_*}(\frac dp+\sigma-2\sigma_1-2)}.
\end{align}
Together with \eqref{a-decay} and \eqref{u-nabla-u-decay},
we conclude that
\begin{equation*}
\begin{aligned}
\|u\|_{\dot{B}_{p,1}^{\sigma}}
&\lesssim
e^{-t}\cX_{p,0}+\cX_{p,0}\int_{0}^te^{-(t-\tau)}(1+\tau)^{-\frac{1}{2s_*}(\sigma+2s_*-1-\sigma_1)}d\tau
\\&\quad+\cX_{p,0}^2\int_{0}^te^{-(t-\tau)}(1+\tau)^{-\frac{1}{2s_*}(\frac dp+\sigma-2\sigma_1-2)}d\tau
\\&\lesssim\cX_{p,0}(1+t)^{-\frac{1}{2s_*}(\sigma+2s_*-1-\sigma_1)}
\end{aligned}
\end{equation*}
with $\sigma\in(\sigma_1+1,\frac dp-2s_*].$ 
Moreover, similar to \eqref{u-decay-2}, one has
\begin{equation*}
\begin{aligned}
\|u\|_{\dot{B}_{p,1}^\sigma}
\lesssim \cX_{p,0}(1+t)^{-\frac{1}{2s_*}(\frac dp-1-\sigma_1)},
\quad
\sigma\in(\frac dp-2s_*,\frac dp+2-s_*].
\end{aligned}
\end{equation*}
The proof of $\eqref{decayau1}$ is finished.

Finally, we are ready to show \eqref{au-high-decay}.
By the embedding $\dot B^{\frac dp}_{p,1}\hookrightarrow L^\infty$, we obtain
\begin{equation}\label{partial-inf-inf}
\begin{aligned}
\|\partial_ta\|_{L_t^
\infty(L^\infty)}
&\lesssim\|\div u\|_{\tilde L_t^
\infty(\dot B^{\frac dp}_{p,1})}
+\|a\|_{\tilde L_t^
\infty(\dot B^{\frac dp}_{p,1})}
\|u\|_{\tilde L_t^
\infty(\dot B^{\frac dp+1}_{p,1})}
\\
&\quad+\|a\|_{\tilde L_t^
\infty(\dot B^{\frac dp+1}_{p,1})}
\|u\|_{\tilde L_t^
\infty(\dot B^{\frac dp}_{p,1})}.
\end{aligned}
\end{equation}
It follows from \eqref{global-Var0}, \eqref{au-dp}, \eqref{au-dp1} and \eqref{partial-inf-inf} that
\begin{equation*}
\|(\nabla a, \nabla u, \partial_ta)\|_{L_t^\infty(L^\infty)}
\lesssim\cE_{p,0}.
\end{equation*}
Due to the smallness of $\cE_{p,0}$, \eqref{lyap} can be written as
\begin{equation*}
\begin{aligned}
\frac12\frac d{dt} \mathcal{L}_j^2(t)
+\mathcal{L}_j^2(t)
&\lesssim\|(\Lambda^{s_*}R_{j}^{1},\Lambda^{s_*}R_{j}^{2},\Lambda R_j^3)\|_{L^2}
\mathcal{L}_j(t).
\end{aligned}
\end{equation*}
This gives rise to 
\begin{equation}\label{au-high-decay-pro}
\begin{aligned}
&\|a\|_{\dot B^{\frac d2+1}_{2,1}}^{h}
+\|u\|_{\dot B^{\frac d2+2-s_*}_{2,1}}^{h}
\\&\quad\lesssim e^{-t}\big(\|a_0\|_{\dot B^{\frac d2+1}_{2,1}}^{h}
+\|u_0\|_{\dot B^{\frac d2+2-s_*}_{2,1}}^{h}\big)
\\&\quad\quad+\int_{0}^te^{-(t-\tau)}\sum_{j\geq J_1-1}2^{j(\frac d2+1-s_*)}\|(\Lambda^{s_*}R_{j}^{1},\Lambda^{s_*}R_{j}^{2}, \Lambda R_j^3)\|_{L^2}d\tau.
\end{aligned}
\end{equation}
Arguing similarly to \eqref{r1}, thanks to \eqref{HLEst}, we deduce 
\begin{equation*}
\begin{aligned}
&\sum_{j\geq J_1-1}2^{j(\frac d2+1-s_*)}\|\Lambda^{s_*}R_{j}^{1}\|_{L^2}
\\&\quad\lesssim
\|\nabla u\|_{\dot B^{\frac dp}_{p,1}}
\|\nabla a\|_{\dot B^{\frac d2}_{2,1}}^{h}
+\| a\|_{\dot B^{\frac dp-1}_{p,1}}^{\ell}
\|u\|_{\dot B^{\frac dp}_{p,1}}^{\ell}
\\&\quad\quad+\|\nabla a\|_{\dot B^{\frac dp}_{p,1}}
\|u\|_{\dot B^{\frac d2+2-s_*}_{2,1}}^{h}
+\|a\|_{\dot B^{\frac dp-1}_{p,1}}^{\ell}
\|u\|_{\dot B^{\frac dp}_{p,1}}^{\ell}.
\end{aligned}
\end{equation*}
By virtue of \eqref{au-fracdp-decay}, \eqref{a-decay-dp1} and \eqref{u-decay-2}, it also holds that
\begin{equation*}
\sum_{j\geq J_{1}}2^{j(\frac{d}{2}+1-s_*)}
\|\Lambda^{s_*}R_{j}^{1}\|_{L^2}
\lesssim \cX_{p,0}^2(1+t)^{-\frac{1}{s_*}(\frac dp-1-\sigma_1)}.
\end{equation*}
Similarly, one can get
\begin{equation*}
\begin{aligned}
 &\sum_{j\geq J_1-1}2^{j(\frac d2+1-s_*)}\|(\Lambda^{s_*}R_{j}^{2},\Lambda R_j^3)\|_{L^2}
\lesssim \cX_{p,0}^2(1+t)^{-\frac{1}{s_*}(\frac dp-1-\sigma_1)}.   \end{aligned}
\end{equation*}
Together with \eqref{au-high-decay-pro}, we finally justify \eqref{au-high-decay}. 
The proof of Theorem \ref{decay-a-u} is thus complete.

\end{proof}
\section{Appendix}\label{sectionappendix}
We state some properties of Besov spaces and related estimates which we repeatedly used in the paper. 
The reader can refer to \cite{BCD} for the details of the first five lemmas. 
The first lemma is devoted to classical Bernstein's inequalities.
\begin{Lemma}[]\label{lemma21}
Let $0<r<R, 1\leq p\leq q\leq \infty$ and $k\in \mathbb{N}$. Then,  for any function $u\in L^p$ and $\lambda_{1}>0$, it holds
\begin{equation}\nonumber
\left\{
\begin{aligned}
&{\rm{Supp}}~ \mathcal{F}(u) \subset \{\xi\in\mathbb{R}^{d}~| ~|\xi|\leq \lambda_{1} R\}\Rightarrow \|D^{k}u\|_{L^q}\lesssim\lambda_{1}^{k+d(\frac{1}{p}-\frac{1}{q})}\|u\|_{L^p},\\
&{\rm{Supp}}~ \mathcal{F}(u) \subset \{\xi\in\mathbb{R}^{d}~|~ \lambda_{1} r\leq |\xi|\leq \lambda_{1} R\}\Rightarrow \|D^{k}u\|_{L^{p}}\sim\lambda_{1}^{k}\|u\|_{L^{p}}.
\end{aligned}
\right.
\end{equation}
\end{Lemma}

The next lemma states the classical interpolation inequalities. 
\begin{Lemma}\label{Classical Interpolation}
Let $1 \leq  p, r, r_1, r_2 \leq \infty$.
\begin{itemize}
\item If $u\in \dot B_{p,r_1}^s \cap \dot B_{p,r_2}^{\tilde{s}} $ and $s\not = \tilde{s}$ then,  $u\in \dot B_{p,r}^{\theta s+(1-\theta)\tilde{s}}$  for all $\theta\in (0, 1)$ and
\begin{equation}\nonumber
\begin{aligned}
\|u\|_{\dot B_{p,r}^{\theta s+(1-\theta)\tilde{s}}}
\lesssim
\|u\|_{\dot B_{p,r}^{s}}^{\theta}
\|u\|_{\dot B_{p,r}^{\tilde{s}}}^{1-\theta}
\end{aligned}
\end{equation}
with 
$$
\frac{1}{r}=\frac{\theta}{r_1}+\frac{1-\theta}{r_2}.
$$
\item If $u\in \dot B_{p,\infty}^s \cap \dot B_{p,\infty}^{\tilde{s}} $ and $s<\tilde{s}$, then,  $u\in \dot B_{p,1}^{\theta s+(1-\theta)\tilde{s}}$  for all $\theta\in (0, 1)$ and
\begin{equation*}
\|u\|_{\dot B_{p,1}^{\theta s+(1-\theta)\tilde{s}}}
\lesssim\frac{1}{\theta(1-\theta)(\tilde{s}-s)}\|u\|_{\dot B_{p,\infty}^{s}}^{\theta}
\|u\|_{\dot B_{p,\infty}^{\tilde{s}}}^{1-\theta}.
\end{equation*}
\end{itemize} 
\end{Lemma}

The following interpolation inequalities for high and low frequencies are also used in this paper.
\begin{cor}\label{Interpolation}
Let $s_1\leq s_2$, $q,r\in[1,+\infty]$, $\theta\in (0,1)$ and $1\leq\alpha_1\leq \alpha\leq \alpha_2\leq \infty$ such that $\frac{1}{\alpha}=\frac{\theta}{\alpha_1}+\frac{1-\theta}{\alpha_2}$. Then we have
\begin{equation*}
\begin{aligned}
&\|u\|_{\tL_T^{\alpha}(\dB_{q,r}^{\theta s_1+(1-\theta)s_2})}^{\ell}
\lesssim\Big(\|u\|_{\tL_T^{\alpha_1}(\dB_{q,r}^{s_1})}^{\ell}\Big)^{\theta}
\Big(\|u\|_{\tL_T^{\alpha_2}(\dB_{q,r}^{s_2})}^{\ell}\Big)^{1-\theta},
\end{aligned}
\end{equation*}
and
\begin{equation*}
\begin{aligned}
&\|u\|_{\tL_T^{\alpha}(\dB_{q,r}^{\theta s_1+(1-\theta)s_2})} ^{h}
\lesssim\Big(\|u\|_{\tL_T^{\alpha_1}(\dB_{q,r}^{s_1})} ^{h}\Big)^{\theta}
\Big(\|u\|_{\tL_T^{\alpha_2}(\dB_{q,r}^{s_2})} ^{h}\Big)^{1-\theta}.
\end{aligned}
\end{equation*}
\end{cor}

The classical product laws  have been used several times.
\begin{Lemma}\label{ClassicalProductLaw}
Let $s>0$, $1\leq p,r\leq \infty$, then,  we have
\begin{equation*}\label{ClassicalProductLawEst1}
\begin{aligned}
&\|uv\|_{\dot{B}^{s}_{p,r}}
\lesssim
\|u\|_{L^\infty}\|v\|_{\dot{B}^{s}_{p,r}}
+\|u\|_{\dot{B}^{s}_{p,r}}\|v\|_{L^\infty}.
\end{aligned}
\end{equation*}
For $d\geq 1$ and $-\min\{d/p, d/p'\}<s\leq d/p$ for $1/p+1/p'=1$, the following inequality holds{\rm:}
\begin{equation*}\label{ClassicalProductLawEst2}
\begin{aligned}
\|uv\|_{\dot{B}^{s}_{p,1}}
\lesssim
\|u\|_{\dot{B}^{\frac{d}{p}}_{p,1}}\|v\|_{\dot{B}^{s}_{p,1}}.
\end{aligned}
\end{equation*}
\end{Lemma}

 Next is the embedding properties and some useful properties.
\begin{Lemma}\label{embedding}
The following statements hold:
\begin{itemize}
\item For any 
$1\leq p\leq q\leq \infty,$ we have the following chain of the continuous embedding:
$$\dot{B}_{p,1}^0\hookrightarrow L^{p}
\hookrightarrow\dot{B}_{p,\infty}^0\hookrightarrow\dot{B}_{q,\infty}^\sigma\quad 
\text{for}\quad \sigma=-d(\frac 1p-\frac 1q)<0.$$
\item  If $\sigma\in\R, 1\leq p_1\leq p_2\leq\infty$ and 
$1\leq r_1\leq r_2\leq\infty$, then
$$
\dot{B}_{p_1,r_1}^{\sigma}\hookrightarrow
\dot{B}_{p_2,r_2}^{\sigma-d(\frac1{p_1}-\frac1{p_2})}.
$$
\item  The space $\dot{B}_{p,1}^{\frac dp}$ is continuously embedded in the set of bounded continuous functions {\rm(}going
to zero at infinity if, additionally, $p<\infty).$
\item Let $\Lambda^{\sigma}$ be defined by $\Lambda^{\sigma}u=(-\Delta)^{\frac{\sigma}{2}}u\triangleq\mathcal{F}^{-1}(|\xi|^{\sigma}\mathcal{F}(u))$ for any $\sigma\in \R$ and $u\in\mathcal{S}_{h}'$, then $\Lambda^{\sigma}$ is an isomorphism from $\dot{B}_{p,r}^{s}$ to $\dot{B}_{p,r}^{s-\sigma}.$
\item 
Let $1\leq p_1,\,p_2,\,r_1,\,r_2\,\leq \infty,\, s_1\in\R,$ and $s_2\in\R$ satisfy
$$s_2<\frac d{p_2},\quad s_2=\frac d{p_2}\,\,\text{and}\,\, r_2=1.$$
Then, the space $\dot{B}_{p_1,r_1}^{s_1}\cap\dot{B}_{p_2,r_2}^{s_2}$ endowed with the norm $\|\cdot\|_{\dot{B}_{p_1,r_1}^{s_1}}+\|\cdot\|_{\dot{B}_{p_2,r_2}^{s_2}}$ is a Banach space.
\end{itemize}
\end{Lemma}


The following lemma is concerned with the classical commutator estimates.
\begin{Lemma}
\label{classical-commutator-estimates}
Let $1\leq p\leq\infty, 1\leq r\leq\infty$
and $-\min(\frac dp,\frac d{p'})<s<\frac dp+1    
(\text{or}~ s=\frac dp+1~  \text{if}~ r=1)$ with $\frac 1p+\frac1{p'}=1$. Then it holds that
$$2^{js}\|\dot S_{j-1}u\partial_{x_i} v_{j}-\dot{\Delta}_{j}(u\partial_{x_i} v)\|_{L^p}
\lesssim c_j\|\nabla u\|_{\dot B^{\frac dp}_{p,1}}
\|v\|_{\dot B^{s}_{p,r}},\quad i=1,2,...,d,
$$
where $c_j\in l^r$ satisfying $\|c_j\|_{l^r}=1$.
\end{Lemma}

We also need the commutator estimates for fractional operator \rm(e.g., see \cite{Bai-2024}\rm).
\begin{Lemma}\cite{Bai-2024}\label{frac-commutator}
\label{frac-commutator-estimates}
Let $\alpha>1.$ Then the following inequality holds true:
$$\|[\Lambda^{\alpha-1},\dot{S}_{j-1}v\cdot\nabla]\dot{\Delta}_j\sigma\|_{L^2}
\leq C2^{j(\alpha-1)}\|\nabla v\|_{L^{\infty}}
\|\dot{\Delta}_j\sigma\|_{L^2}.$$
\end{Lemma}
 The following lemma is very useful in high frequencies.
\begin{Lemma}\cite{Barat-Danchin,Barat-Danchin2023}
\label{barat-commutator-estimates}
Let $2\leq p\leq 4$ and $s > 0$. Define $p^*=\frac{2p}{p-2}${\rm(}$p^*=\infty, p=2${\rm)}. For $j\in \Z$, denote
$R_j=\dot{S}_{j-1} w \dot\Delta_j z-\dot\Delta_j(wz)$. 
There exists a constant $C$  such that
\begin{equation*}
\begin{aligned}
&\sum\limits_{j\geq J_1-1} 2^{sj}\|R_j\|_{L^2} \\
&\quad \leq
C\bigg(\|\nabla w\|_{\dB_{p,1}^{\frac{d}{p}}}\|z\|_{\dB_{2,1}^{s-1}}^{h}
+2^{(s-\eta_1)J_1}\|z\|_{\dB_{p,1}^{\frac{d}{p}-\frac{d}{p^*}}}^{\ell}
\|w\|_{\dB_{p,1}^{\eta_1}}^{\ell}
\\&\qquad+\|z\|_{\dB_{p,1}^{\frac{d}{p}-k}}\|w\|_{\dB_{2,1}^{s+k}}^{h}
+2^{(s-\eta_2-1)J_1}\|z\|_{\dB_{p,1}^{\eta_2}}^{\ell}
\|\nabla w\|_{\dB_{p,1}^{\frac{d}{p}-\frac{d}{p^*}}}^{\ell}
\bigg),
\end{aligned}
\end{equation*}
for any $k\geq0, \eta_1\geq s$ and $\eta_2\in\R.$
\end{Lemma}
To establish the evolution of the $\dot B_{p,\infty}^{\sigma_1}$-norm for low frequencies, we also need the product laws in the case that the third index of Besov space is equal to $\infty$.
\begin{Lemma}\cite{xin-xu}\label{negative product estimate}
Let $-\frac dp\leq\sigma_1<\frac dp$ and $p\geq2$. 
Then it holds that
\begin{equation*}
\begin{aligned}
\|fg\|_{\dot{B}_{p,\infty}^{\sigma_1}}
\lesssim\|f\|_{\dot{B}_{p,1}^{\frac dp}}
\|g\|_{\dot{B}_{p,\infty}^{\sigma_1}}.
\end{aligned}
\end{equation*}
\end{Lemma}

Here are the lower bounds of an integral involving the fractional power dissipative term.
\begin{Lemma}\cite{Jiahong Wu}
\label{low-fra-Lp}
Assume either $0\leq\alpha$ and $p=2$ or $0\leq\alpha\leq1$ and $2<p<\infty.$
If $f\in\cC^2(\R^{d})$
decays sufficiently fast at infinity and satisfies
$$\text{Supp}~\hat{f}\subset \Big{\{}\xi\in \R^{d}~|~K_12^j\leq |\xi|\leq K_22^j \Big{\}}$$
for some $K_1,K_2>0$ and some integer $j,$ then
$$D(f)\equiv\int_{\R^{d}}|f|^{p-2}f\cdot(-\Delta)^\alpha f\,dx\geq C2^{2\alpha j}\|f\|_{L^p(\R^{d})}^p,$$
where $C$ is a constant depending on $d,p,K_1$ and $K_2$ only.
\end{Lemma}
Next one is a standard lemma of some differential inequality.
\begin{Lemma}\cite{xu-2017}\label{difference-delta}
Let $X:[0,T]\rightarrow\R_+$ be a continuous function such that $X^p$ is
differentiable for some $p\geq1$ and satisfies
$$\frac1p\frac d{dt}X^p(t)+aX^p(t)\leq AX^{p-1}(t)$$
for some constant $a\geq0$ and measurable function
$A:[0,T]\rightarrow\R_+.$
Then it holds that
$$X(t)+a\int_0^tX(\tau)\,d\tau\leq X(0)+\int_0^tA(\tau)d\tau.$$
\end{Lemma}

\vspace{1mm}

\bigbreak

\noindent
\textbf{Funding.}  L.-Y. Shou is supported by National Natural Science Foundation of China (12301275) and the project funded by China Postdoctoral Science Foundation  (2023M741694). J. Xu is partially supported by National Natural Science Foundation of China (12271250, 12031006) and the Fundamental Research Funds for the Central Universities, NO. NP2024105.

\vspace{2mm}

\noindent
\textbf{Conflict of interest.} The authors do not have any possible conflict of interest.

\vspace{2mm}

\vspace{2mm}

\noindent
\textbf{Data availability statement.}
 Data sharing not applicable to this article as no data sets were generated or analysed during the current study.

\end{document}